\documentclass[a4paper]{article}
\usepackage{amsmath}
\usepackage[makeroom]{cancel}
\usepackage{amsfonts}
\usepackage{amssymb}
\usepackage{amsthm}
\usepackage{graphicx}
\usepackage[all,2cell]{xy}
\usepackage{todonotes}

\newtheorem{theorem}{Theorem}[section]

\newtheorem{cor}[theorem]{Corollary}

\newtheorem{definition}[theorem]{Definition}
\newtheorem{ex:}[theorem]{Example}

\newtheorem{lemma}[theorem]{Lemma}

\newtheorem{prop}[theorem]{Proposition}
\newtheorem{remark}[theorem]{Remark}

\numberwithin{equation}{section}

\newcommand{\Rr}{\mathbb R}

\newcommand{\rest}[1]{\Big{\vert}_{#1}}

\renewcommand{\d}{\mathrm d}

\let\epsilon\varepsilon
\newcommand{\LA}{\ensuremath{\mathcal{LA}}}     

\newcommand{\XX}{\mathfrak X}           
\newcommand{\G}{\mathcal G}             
\newcommand{\F}{\ensuremath{\mathcal F}} 
\newcommand{\Lie}{\mathcal{L}}          
\renewcommand{\gg}{\mathfrak{g}}        
\newcommand{\hh}{\mathfrak{h}}          
\newcommand{\kk}{\mathfrak{c}}          
\newcommand{\ggl}{\mathfrak{gl}}        
\renewcommand{\ll}{\mathfrak{l}}        
\DeclareMathOperator{\coker}{coker}     
\DeclareMathOperator{\Img}{Im}          
\DeclareMathOperator{\Ad}{Ad}           
\DeclareMathOperator{\ad}{ad}           

\newcommand{\ddt}[1]{\frac{d}{d #1}\Big{\vert}_{#1 =0}}			   

\begin{document}

\title{Remarks on the structure and integrability of LA-groups}
\author{Camilo Angulo \\
Mathematics Institute, Georg-August-University G\"ottingen, \\
Bunsenstra{\ss}e 3-5, G\"ottingen 37073, Germany \\
ca.angulo951@gmail.com}
\date{2026}
\maketitle

\begin{abstract}
We study the structure of an LA-group identifying its underlying VB-group with a representation up to homotopy. We show that the Lie algebroid structure is determined by a complementary action up to homotopy of the Lie algebra of units. We identify the equations that the representation and the action need to verify in order to assemble into an LA-group, establishing an equivalence between LA-groups and LA-matched pairs. As an application, we catalog some extreme examples of representations and actions and comment on their integrability. 
\end{abstract}





\section*{Introduction}
An LA-groupoid is a groupoid object internal to the category of Lie algebroids. LA-groupoids can be obtained by applying the Lie functor to one of Ehresmann's double Lie groupoids; however, whether every LA-groupoid is of this form remains an open problem \cite{LucaPhD}. 
The elegant theory of integration developed in \cite{CF2} does not account for certain topological peculiarities that appear in the categorified theory, and therefore cannot be applied directly. 
Consequently, progress has relied largely on ad hoc strategies: understanding the structure of double Lie algebras for the integration of the cotangent groupoid of a Poisson-Lie group to a double Lie group~\cite{LuWein}, later extended to Poisson homogeneous spaces~\cite{BIL} and to certain double products of groupoids~\cite{2ndOrdGeom,Daniel1}; the use of regular monoid actions for the integration of VB-algebroids to VB-groupoids \cite{BCD}; and the adaptation of the van Est strategy for the integration of strict Lie 2-algebras to strict Lie 2-groups \cite{ZhuInt2Alg,Angulo:2021}. From a higher-categorical viewpoint, double Lie groupoids can be regarded as bisimplicial objects satisfying suitable filling conditions, which, after applying of the codiagonal functor, yield local Lie 2-groupoids \cite{Mehta_Tang:2011}. This outlines how the integration of LA-groupoids also fits into the current trends in higher Lie theory \cite{RyvkZhu2026,CabreraDelHoyo2026}.

\noindent Double Lie groupoids are not only of intrinsic mathematical interest; they have recently played a central r\^ole in resolving a long-standing conjecture concerning the existence of generalized K\"ahler potentials expressing the metric locally in terms of a single real scalar function \cite{Daniel-Marco}. 
This was accomplished by describing the underlying holomorphic structure of a generalized K\"ahler manifold as a diagram in the 2-category of symplectic double Lie groupoids. 
This renewed interest in double structures prompted us to revisit the integrability problem for LA-groupoids.

\noindent In view of the importance that understanding the Lie theory of groups and algebras played in the integration of Lie algebroids, we focus our attention on studying LA-groups, that is, LA-groupoids whose space of objects is a Lie algebra. 
Using the correspondence established in \cite{VB&Reps}, we express the groupoid structure using a representation up to homotopy of the base Lie group. 
By writing down the equations that this representation needs to satisfy for the structural maps to be compatible with the algebroid structure, we identify a complementary \textbf{action up to homotopy} of the Lie algebra of units. 
This complementary action determines the Lie algebroid structure on the vector bundle of arrows (see Definition~\ref{Def:AUTH} and Proposition~\ref{Prop:MainAUTH}). 
The appearance of this complementary structure is a somewhat desired feature as any satisfactory infinitesimal description must reflect the inherent symmetry of double Lie groupoids.
In addition, by identifying the relations that the representation and the action must satisfy in order to assemble into an LA-group, we generalize the main result of~\cite{MackJotzRaj} and establish a correspondence between LA-groups and \textbf{LA-matched pair} (see Definition~\ref{Def:LAMtchdPair} and Theorem~\ref{Theo:MainEqceSuite}). 
As an application, we catalog some extreme classes of examples of representations and actions and comment on their integrability. 
In particular, in Theorem~\ref{Theo:IntGps}, we prove that whenever the anchor of the algebroid of arrows of an LA-group has constant rank, there is a double Lie group integrating it; 
and, in Section~\ref{ssec-TheWorld}, we prove that a similar result holds under mild completeness assumptions for \emph{vacant} LA-groups.

\noindent The paper is organized as follows: Section~\ref{sec:1} contains preliminaries and introduce our main tool, Theorem~\ref{Theo:MainEqce}. Subsection~\ref{ssec:2} collects results concerning the groupoid structure underlying an LA-group. Section~\ref{sec:3} derives the compatibility equations required for a representation up to homotopy to build a groupoid whose structural maps are Lie algebroid morphisms, first those implied by their being anchored in Subsection~\ref{ssec-anchors} and then those implied by their respecting the brackets in Subsection~\ref{ssec-brackets}. Section~\ref{sec:4} collects applications related to the structure: Subsection~\ref{ssec-LAMtchPrs} extends Theorem~\ref{Theo:MainEqce} to Theorem~\ref{Theo:MainEqceSuite} by identify the underlying \emph{action up to homotopy} of the Lie algebra of untis and reconstructing the LA-group using a complementing representation up to homotopy; then, Subsection~\ref{ssec-LAStr} lays down some results concerning the algebroid structure of the space of arrows, highlighting thus some of the symmetric features of an LA-group. Lastly, Section~\ref{sec:5} collects applications related to the integrability of some extreme classes of examples of LA-groups.

${}$

\noindent\textbf{Acknowledgments.} The author would like to thank the hospitality of both IMPA and Georg-August Universit\"at G\"ottingen, where this work was carried out. The author is also grateful to Stefano Ronchi for carefully reading a preliminary version of this article and for helping in improving its presentation. Lastly, we would like to express our gratitude to the anonymous referee for the observation in Subsection~\ref{sss-vanishingCurvatures} and for their thoughtful comments.

${}$

\noindent\textbf{Notational conventions:} If $E\overset{\pi}\to B$ is a vector bundle and $b\in B$, we convene that $E_b$ is the fibre over $b$, that is the vector space $\pi^{-1}(b)$. Similarly, if $F:E\longrightarrow E'$ is a map of vector bundles covering $f:B\longrightarrow B'$, then $F_p:E_p\longrightarrow E'_{f(p)}$ denotes the linear map at the level of fibres. For any given groupoid $G\rightrightarrows M$ and any point $p\in M$, we write $\G_p(G)=\lbrace g\in G\mid t(g)=p=s(g)\rbrace$ for the \textbf{isotropy group} at $p$, and $O^G_p=\lbrace x\in M\mid\exists g\in G:t(g)=x, s(g)=p\rbrace$ for the \textbf{orbit} through $p$; furthermore, if $G$ is a Lie groupoid, $A_G$ stands for its Lie algebroid. For any given Lie algebroid $A\to M$ with anchor $\rho:A\longrightarrow TM$ and any point $p\in M$, we write $\gg_p(A)=\ker(\rho_p)$ for the \textbf{isotropy Lie algebra} at $p$, and $L_p^A$ for the \textbf{leaf} through $p$ integrating the distribution $\F^A:=\rho(\Gamma(A))\leq\XX(M)$. Given a map $f:M'\longrightarrow M$, $f^*A$ is the pull-back vector bundle $M'\times_MA$ and, whenever $df$ is transversal to the anchor, $f^!A$ the pull-back Lie algebroid $TM'\times_{f^*TM}f^*A$.


\section{LA -groups and their associated representations up to homotopy}\label{sec:1}
Throughout, we work with the LA-group $A\rightrightarrows\hh$ and we denote its base $G$. 
Let $e\in G$ be the identity element, 
then, as the source map $s:A\longrightarrow\hh$ is fibre-wise surjective, 
there is a canonically split short exact sequence of vector spaces
\begin{eqnarray}\label{Eq:unitSplit}
\xymatrix{
(0) \ar[r] & \kk \ar[r] & A_e \ar[r]_{s_e} & \hh \ar[r]\ar@/_1pc/[l]_u & (0),
}
\end{eqnarray}
where $\kk :=\ker (s_e)\leq A_e$ is the \textbf{core} and $u$ is the unit map. 
Right multiplication yields an isomorphism $\ker(s)\cong G\times\kk$. Indeed, if $(a_0,a_1)\in A_g\oplus_{\hh}A_h$, we write $a_0\Join a_1:=m_{(g,h)}(a_0,a_1)\in A_{gh}$ and with this notation, the isomorphism and its inverse are given by $\ker(s_g)\ni a\mapsto(g,a\Join 0_{g^{-1}})$ and $(g,z)\mapsto z\Join 0_g$, respectively. 
As a map of vector bundles, $s$ covers the constant map $\bar{\ast}:G\longrightarrow\ast$ and $\bar{\ast}^*\hh\cong G\times\hh$; hence, the \textbf{core sequence} is isomorphic to
\begin{eqnarray}\label{Eq:CoreSeq}
\xymatrix{
(0) \ar[r] & G\times\kk \ar[r] & A \ar[r] & G\times\hh \ar[r] & (0),
}
\end{eqnarray}
from which it follows that $A$ is necessarily trivial as a vector bundle. A splitting $\sigma:G\times\hh\longrightarrow A$ of \eqref{Eq:CoreSeq} is said to \textbf{extend the unit} if $\sigma_e(y)=u(y)$, for all $y\in\hh$. Let $\sigma$ be a unit extending splitting and consider the induced isomorphism $A\cong G\times(\kk\oplus\hh)$. Then, $G\times(\kk\oplus\hh)$ inherits a VB-groupoid structure whose source, target and unit maps are respectively
\begin{eqnarray}\label{Eq:StrMaps}
s(g;z,y)=y, & t(g;z,y)=\partial z+\Delta^\hh_gy, & u(y)=(e;0,y),
\end{eqnarray}
where $\partial z:=t_e(z)$ and $\Delta^\hh_gy:=t_g(\sigma_g(y))$.
For the multiplication, first notice that $s(g;z_0,y')=t(h;z_1,y)$ if and only if $y'=\partial z_1+\Delta^\hh_{h}y$; therefore, the space of composable arrows $A\times_{\hh}A$ is isomorphic to $G^2\times(\kk^2\oplus\hh)$. Under this identification, one concludes that the multiplication is given by
\begin{eqnarray}\label{Eq:StrMult}
m(g,h;z_0,z_1,y):=(gh;z_0+\Delta^\kk_gz_1-\Omega_{(g,h)}(y),y),
\end{eqnarray}
where $\Delta^\kk_g(z):=\sigma_g(t_e(z))\Join z\Join 0_{g^{-1}}$ and 
\begin{align*}
    \Omega_{(g,h)}(y) & :=\big{(}\sigma_{gh}(y)-\sigma_g(\Delta^\hh_hy)\Join\sigma_h(y)\big{)}\Join 0_{(gh)^{-1}}.
\end{align*}
Recall that a \textbf{representation up to homotopy} (RUTH for short) of a Lie group $G$ on the graded vector space $V_{1}\oplus V_0$ is a 4-tuple $(\partial,\Delta^0,\Delta^1,\Omega)$ that consists of a linear map $\partial:V_{1}\longrightarrow V_0$, smooth maps $\Delta^k:G\longrightarrow V_{k}^*\otimes V_k$ such that 
$\Delta^k_e=\textnormal{Id}_{V_k}$, and a smooth map $\Omega:G^2\longrightarrow V_{0}^*\otimes V_{1}$ that vanishes on $G\times\lbrace e\rbrace$ and $\lbrace e\rbrace\times G$, further subject to the following relations:
\begin{align*}
    \partial\circ\Delta^1 & =\Delta^0\circ\partial, & \Delta^0_{gh}-\Delta^0_g\circ\Delta^0_h & =\partial\circ\Omega_{(g,h)}, \\
    \Delta^1_g\circ\Omega_{(h,k)}-\Omega_{(gh,k)} & +\Omega_{(g,hk)}-\Omega_{(g,h)}\circ\Delta^0_k=0, & \Delta^1_{gh}-\Delta^1_g\circ\Delta^1_h & =\Omega_{(g,h)}\circ\partial,
\end{align*}
for all $g,h,k\in G$. 
\begin{theorem}{\cite{VB&Reps}}\label{Theo:MainEqce}
    Given a VB-group $A\rightrightarrows\hh$ over $G$ with core $\kk$ and a unit extending splitting of its core sequence~\eqref{Eq:CoreSeq}, $(\partial,\Delta^\hh,\Delta^\kk,\Omega)$ is a RUTH of $G$ on $\kk[1]\oplus\hh$. Moreover, the structural maps~\eqref{Eq:StrMaps} and~\eqref{Eq:StrMult} define a VB-groupoid structure on $G\times(\kk\oplus\hh)$ if and only if $(\partial,\Delta^\hh,\Delta^\kk,\Omega)$ is a RUTH.
\end{theorem}

In the sequel, we study the LA-group structure through the lens of an associated RUTH. 
Before doing that, though, we use Theorem~\ref{Theo:MainEqce} to study the groupoid structure of $A\rightrightarrows\hh$.


\subsection{The groupoid structure}\label{ssec:2}
Given a VB-group $A\rightrightarrows\hh$ and a splitting of~\eqref{Eq:CoreSeq}, one finds an isomorphism of $A$ onto $G\times(\kk\oplus\hh)$ together with the structural maps~\eqref{Eq:StrMaps} and~\eqref{Eq:StrMult}. Under this isomorphism, one can express the isotropy group at $y_0\in\hh$, as
\begin{align*}
    \G_{y_0}(A) & \cong\lbrace(g;z,y)\in G\times(\kk\oplus\hh)\mid t(g;z,y)=y_0=s(g;z,y)\rbrace \\
            & =\lbrace(g;z)\in G\times\kk\mid\partial z=y_0-\Delta^\hh_gy_0\rbrace\times\lbrace y_0\rbrace.
\end{align*}
Because $(\partial,\Delta^\hh,\Delta^\kk,\Omega)$ is a RUTH, for each $g,h\in G$, $\Delta^\hh_{gh}y-\Delta^\hh_g(\Delta^\hh_hy)=\partial\Omega_{(g,h)}(y)$, the \emph{quasi-action} $\Delta^\hh$ descends to an honest action $\Delta^0:G\longrightarrow\textnormal{GL}(\coker(\partial))$. Indeed, for $y+\Img(\partial)\in\coker(\partial)$, $\Delta^0_g(y+\Img(\partial)):=\Delta^\hh_gy+\Img(\partial)$. This action is well-defined: If $y'-y=\partial z$, $$\Delta^\hh_gy'=\Delta^\hh_gy+\Delta^\hh_g(\partial z)=\Delta^\hh_gy+\partial(\Delta^\kk_g z).$$
Note that the structural map $\partial=t_e\rest{\kk}$ does not depend of the splitting, so neither does $\coker(\partial)$. In fact, any two splittings of~\eqref{Eq:CoreSeq} induce the same action $\Delta^0$.  
\begin{lemma}\label{lemma:Dela0Indep}
    $\Delta^0$ is independent of the splitting.
\end{lemma}
\begin{proof}
    Let $\sigma^0$ and $\sigma^1$ be two splittings of~\eqref{Eq:CoreSeq}. Then, for all $y\in\hh$ and all $g\in G$, $\sigma^1_g(y)-\sigma^0_g(y)\in\ker(s_g)$ and there exists a unique $z\in\kk$ such that $\sigma^1_g(y)-\sigma^0_g(y)=z\Join0_g$. As a consequence, the two quasi-actions descend to the same action on $\coker(\partial)$, 
    $$t_g(\sigma^1_g(y))-t_g(\sigma^0_g(y))=t_g(z\Join 0_g)=\partial z.$$
\end{proof}
Similarly, because for each $g,h\in G$, $\Delta^\kk_{gh}z-\Delta^\kk_g(\Delta^\kk_hz)=\Omega_{(g,h)}(\partial z)$; hence, the \emph{quasi-action} $\Delta^\kk$ restricts to an honest action $\Delta^1:G\longrightarrow\textnormal{GL}(\ker(\partial))$. As in Lemma~\ref{lemma:Dela0Indep}, since the structural map $\partial=t_e\rest{\kk}$ does not depend of the splitting, neither does $\ker(\partial)$ and any two splittings of~\eqref{Eq:CoreSeq} induce the same action $\Delta^1$.  
\begin{lemma}\label{lemma:Dela1Indep}
    $\Delta^1$ is independent of the splitting.
\end{lemma}
\begin{proof}
    Let $\sigma^0$ and $\sigma^1$ be two splittings of~\eqref{Eq:CoreSeq}. Then, for all $z\in\ker(\partial)$ and all $g\in G$, 
    $$\sigma^1_g(\partial z)\Join z\Join0_{g^{-1}}-\sigma^0_g(\partial z)\Join z\Join0_{g^{-1}}=0_g\Join z\Join0_{g^{-1}}-0_g\Join z\Join0_{g^{-1}}=0_e.$$
\end{proof}

In the statement of the following lemma, we write $\delta:C^n(G,\kk)\longrightarrow C^{n+1}(G,\kk)$ for the \textbf{simplicial differential} of a group $G$ with values in a quasi-representation $\Delta^\kk$. Note that in spite of our calling it a differential, $\delta$ does not square to zero. For instance, given a 1-cochain $Z\in C(G,\kk)$ and $g,h,k\in G$, 
\begin{align}\nonumber 
    (\delta^2Z)(g,h,k) & =\Delta^\kk_g(\delta Z(h,k))-\delta Z(gh,k)+\delta Z(g,hk)-\delta Z(g,h) \\ \label{Eq:NoDiff}
        & =\Delta^\kk_g\Delta^\kk_hZ_k-\Delta^\kk_{gh}Z_k=-\Omega_{(g,h)}(\partial Z_k)\neq 0.
\end{align}
\begin{prop}\label{Prop:GpdIsotropies}
    Let $A\rightrightarrows\hh$ be a VB-group over $G$ with core $\kk$ and let $y\in\hh$. Then, there is a short exact sequence of Lie groups
\begin{eqnarray}\label{Eq:IsoCoreSeq}
\xymatrix{
(1) \ar[r] & \ker(\partial) \ar[r] & \G_y(A) \ar[r]^{\pi\quad} & G^{\Delta^0}_{y+\Img(\partial)} \ar[r] & (1),
}
\end{eqnarray}
where $G^{\Delta^0}_{y+\Img(\partial)}$ is the \emph{stabilizer subgroup} of $y+\Img(\partial)\in\coker(\partial)$ for the action $\Delta^0$, and $\pi$ is the restriction of the bundle map $A\to G$.
Moreover, if the sequence~\eqref{Eq:IsoCoreSeq} splits, $\G_y(A)$ is isomorphic to the semi-direct product of $G^{\Delta^0}_{y+\Img(\partial)}$ and $\ker(\partial)$ with respect to the action $\Delta^1$ twisted by a 2-cocycle. More precisely, since any splitting is determined by a map $Z\in C(G^{\Delta^0}_{y+\Img(\partial)},\kk)$, the 2-cocycle is explicitly given by $\delta Z-\Omega(y)\in C^2(G^{\Delta^0}_{y+\Img(\partial)},\ker(\partial))$.
\end{prop}
\begin{proof}
    To see that $\pi$ is well-defined, we prove that if $a\in\G_y(A)$, $\pi(a)\in G^{\Delta^0}_{y+\Img(\partial)}$. Fixed a unit extending a splitting of~\eqref{Eq:CoreSeq}, $a$ corresponds uniquely to a $(g;z,y)\in G\times(\kk\oplus\hh)$ such that $\partial z=y-\Delta_g^\hh y$, then $$\Delta^0_g(y+\Img(\partial))=\Delta_g^\hh y+\Img(\partial)\overset{!}{=}y-\partial z+\Img(\partial)=y+\Img(\partial).$$
    That $\pi$ is onto is straightforward. Indeed, if $g\in G^{\Delta^0}_{y+\Img(\partial)}$, then $\Delta^0_g(y+\Img(\partial))=\Delta_g^\hh y+\Img(\partial)=y+\Img(\partial)$; therefore, there exists a $z\in\kk$ such that $y=\Delta^\hh_gy+\partial z$, and there is a unique element in $\G_y(A)$ corresponding to $(g;z,y)$. Now, under the isomorphism, 
    \begin{align*}
    \ker(\pi) & =\lbrace a\in\G_y(A)\mid\pi(a)=e\rbrace\cong\lbrace(e;z,y)\in G\times(\kk\oplus\hh)\mid\Delta^\hh_ey+\partial z=y\rbrace \\
              & =\lbrace e\rbrace\times\lbrace z\in\kk\mid\partial z=0\rbrace\times\lbrace y\rbrace\cong\ker(\partial).
    \end{align*}
    The last statement follows now, as a splitting under the isomorphism is given by $g\longmapsto (g;Z_g,y)$ for some $Z\in C(G^{\Delta^0}_{y+\Img(\partial)},\kk)$ such that $\partial Z_g=y-\Delta^\hh_gy$. Any such splitting induces a representation $\rho^Z$ by conjugation in $\G_y(A)$ and a cocycle $\mu^Z\in C^2(G^{\Delta^0}_{y+\Img(\partial)},\ker(\partial))$. Letting $g\in G^{\Delta^0}_{y+\Img(\partial)}$ and $v\in\ker(\partial)$ and using formula~\eqref{Eq:StrMult}, one computes
    \begin{align*}
        \rho^Z_g(v) & =(g;Z_g,y)\Join(e;v,y)\Join(g;Z_g,y)^{-1} \\
            & =\big(g;Z_g+\Delta^\kk_gv-\Omega_{(g,e)}(y),y\big)\Join\big(g^{-1};-\Delta^\kk_{g^{-1}}Z_g+\Omega_{(g^{-1},g)}(y),y\big) \\
            & =\big(e;Z_g+\Delta^1_gv+\Delta^\kk_g(-\Delta^\kk_{g^{-1}}Z_g+\Omega_{(g^{-1},g)}(y))-\Omega_{(g,g^{-1})}(y),y\big) \\
            & =(e;\Delta^1_gv+\Omega_{(g,g^{-1})}(\partial Z_g)+\Delta^\kk_g\Omega_{(g^{-1},g)}(y)-\Omega_{(g,g^{-1})}(y),y\big)=(e;\Delta^1_gv,y).
    \end{align*}
    Here, we used the cocycle equation for the curvature $\Omega$ evaluated at $(g,g^{-1},g)$ and that both $\Omega_{(g,e)}$ and $\Omega_{(e,g)}$ vanish.
    The value of the cocycle $\mu^Z$ is given for $g,h\in G^{\Delta^0}_{y+\Img(\partial)}$ by $\mu^Z(g,h)=(g;Z_g,y)\Join(h;Z_h,y)\Join(gh;Z_{gh},y)^{-1}$, which equals
    \begin{eqnarray*}
        \big(gh;Z_g+\Delta^\kk_gZ_h-\Omega_{(g,h)}(y),y\big)\Join\big((gh)^{-1};-\Delta^\kk_{(gh)^{-1}}Z_{gh}+\Omega_{((gh)^{-1},gh)}(y),y\big). 
    \end{eqnarray*}
    The $\kk$-entry of the latter is then computed to be
    \begin{align*}
        Z & _g+\Delta^\kk_gZ_h-\Omega_{(g,h)}(y)+\Delta^\kk_{gh}(-\Delta^\kk_{(gh)^{-1}}Z_{gh}+\Omega_{((gh)^{-1},gh)}(y))-\Omega_{(gh,(gh)^{-1})}(y) \\
        = & \delta Z(g,h)-\Omega_{(g,h)}(y)+\Omega_{(gh,(gh)^{-1})}(\partial Z_{gh})+\Delta^\kk_{gh}\Omega_{((gh)^{-1},gh)}(y)-\Omega_{(gh,(gh)^{-1})}(y), 
    \end{align*}
    and the result follows using the same relation as before.
\end{proof}
\begin{remark}\label{Rmk:WhyNotOmega}
    Reading formula~\eqref{Eq:StrMult}, one would be tempted to guess that $-\Omega(y)$ is the 2-cocycle twisting the semi-direct product in Proposition~\ref{Prop:GpdIsotropies}; however, $\Omega$ has no reason to be $\ker(\partial)$-valued, nor is it a true cocycle. Note that it is precisely after adding $\delta Z$ that these ailments are remedied as, by definition, for all $g,h\in G$,
    \begin{align*}
        \partial\Omega_{(g,h)}(y) & =\Delta^\hh_{gh}y-\Delta^\hh_g\Delta^\hh_hy=y-\partial Z_{gh}-\Delta^\hh_g(y-\partial Z_h) \\
            & =y-\partial Z_{gh}-(y-\partial Z_g)+\partial \Delta^\kk_gZ_h=\partial(\delta Z(g,h)).
    \end{align*}
    Similarly, for all $g,h,k\in G$, one computes
    \begin{align*}
        \delta(\Omega(y))(g,h,k) & =\Delta^1_g\Omega_{(h,k)}(y)-\Omega_{(gh,k)}(y)+\Omega_{(g,hk)}(y)-\Omega_{(g,h)}(y) \\
            & =\Delta^1_g\Omega_{(h,k)}(y)-\Omega_{(gh,k)}(y)+\Omega_{(g,hk)}(y)-\Omega_{(g,h)}(\Delta^\hh_ky+\partial Z_k);
    \end{align*}
    hence, out of the cocycle equation for the curvature, one sees that $\delta(\Omega(y))(g,h,k)=-\Omega_{(g,h)}(\partial Z_k)$, which matches Eq.~\eqref{Eq:NoDiff}.
\end{remark}
Recall that, in a Lie groupoid, the source fibre over a point $p$ is a principal bundle over the orbit through $p$ with structure group the isotropy at $p$. In so, for $y_0\in\hh$, $s^{-1}(y_0)\cong\lbrace(g;z,y)\in G\times(\kk\oplus\hh)\mid y=y_0\rbrace=G\times\kk\times\lbrace y_0\rbrace$ is a $\G_{y_0}(A)$-principal bundle over $O^A_{y_0}$. Further note that one can think of $\kk$ as a $\ker(\partial)$-principal bundle over $\Img(\partial)$, and similarly, of $G$ as a $G^{\Delta^0}_{y+\Img(\partial)}$-principal bundle over $O^{\Delta^0}_{y+\Img(\partial)}$, the orbit through ${y+\Img(\partial)}\in\coker(\partial)$ under the action $\Delta^0$. Ultimately, this yields the following result.
\begin{prop}\label{Prop:GpdOrbits}
    Let $A\rightrightarrows\hh$ be a VB-group and let $y\in\hh$. If $O^{\Delta^0}_{y+\Img(\partial)}$ is the orbit of ${y+\Img(\partial)}\in\coker(\partial)$ under $\Delta^0$, then $O^A_y$ is an $\Img(\partial)$-principal bundle over $O^{\Delta^0}_{y+\Img(\partial)}$. Moreover, 
\begin{eqnarray}\label{Eq:Orbit}
O_y^A=\cup O^{\Delta^0}_{y+\Img(\partial)}=\lbrace y'\in\hh\mid y'+\Img(\partial)\in O^{\Delta^0}_{y+\Img(\partial)}\rbrace. 
\end{eqnarray}
\end{prop}
\begin{proof}
    We start by proving that Eq.~\eqref{Eq:Orbit} holds. Splitting~\eqref{Eq:CoreSeq}, one computes: \\
    $(\subseteq)$ Let $y'\in O^A_y$. Then, there exists $(g;z)\in G\times\kk$ such that $y'=\Delta^\hh_gy+\partial z$. Consequently, $y'+\Img(\partial)=\Delta^\hh_gy+\partial z+\Img(\partial)=\Delta^\hh_gy+\Img(\partial)=\Delta^0_g(y+\Img(\partial))$, and $y'+\Img(\partial)\in O^{\Delta^0}_{y+\Img(\partial)}$. \\
    $(\supseteq)$ Let $y'\in\cup O^{\Delta^0}_{y+\Img(\partial)}$, then $y'+\Img(\partial)\in O^{\Delta^0}_{y+\Img(\partial)}$. In other words, there exists a $g\in G$ such that $y'+\Img(\partial)=\Delta^0_g(y+\Img(\partial))=\Delta^\hh_gy+\Img(\partial)$; consequently, there exists a $z\in\kk$ such that $y'=\Delta^\hh_gy+\partial z$, and $y'\in O_y^A$. \\
    This shows that the quotient map $O^A_y\to O^{\Delta^0}_{y+\Img(\partial)}:y'\mapsto y'+\Img(\partial)$ is well-defined. This map is clearly surjective and its fibres are  $\Img(\partial)$.
\end{proof}


\section{Consequences of being LA}\label{sec:3}
Recall that a morphisms of vector bundles $F:A'\longrightarrow A$ between Lie algebroids $A'$ and $A$ covering $f:M'\longrightarrow M$ is said to be \textbf{anchored} if $\rho^{A}\circ F=df\circ\rho^{A'}$. An anchored map $F$ is a \textbf{Lie algebroid morphism} if the map
\begin{eqnarray*}
\xymatrix{
F^!:A' \ar[r] & f^!A:a \ar@{|->}[r] & (\rho^{A'}(a),F(a)),
}
\end{eqnarray*}
induces a map of Lie algebras at the level of sections. More precisely, if the map 
\begin{eqnarray*}
\xymatrix{
F^!:\Gamma(A') \ar[r] & \Gamma(f^!A):a \ar@{|->}[r] & (\rho^{A'}(a),F_*(a)),
}
\end{eqnarray*}
where $\rho^{A'}(a)\in\XX(M')$ and $F_*(a)\in\Gamma(f^*A)\cong C(M')\otimes\Gamma(A)$, defines a Lie algebra homomorphism.
Given an LA-group $A\rightrightarrows\hh$ over $G$ with core $\kk$ and a unit extending splitting $\sigma$ of~\eqref{Eq:CoreSeq}, there is an isomorphism $A\cong G\times(\kk\oplus\hh)$. In this section, we study the consequences of demanding that the structural maps~\eqref{Eq:StrMaps} and~\eqref{Eq:StrMult} be Lie algebroid morphisms.


\subsection{Compatibility with the anchors}\label{ssec-anchors}
In this subsection, we study the consequences of the structural maps being anchored. That the source and target maps, $s$ and $t$ are anchored is inconsequential. Recall from~\eqref{Eq:unitSplit} that $A_e$ is canonically isomorphic to $\kk\oplus\hh$. If $u$ is an anchored map, it follows that $\rho_e\rest{u(\hh)}\equiv 0$. The following is an easy consequence. 
\begin{lemma}\label{lemma:uAnchord}
If $A\rightrightarrows\hh$ is an LA-group with core $\kk$, the isotropy Lie algebra at the identity element $\gg_e(A)$ fits in a short exact sequence of vector spaces
   \begin{eqnarray}\label{Eq:IsotrpyAlgExt}
\xymatrix{
(0) \ar[r] & \mathfrak{k} \ar[r] & \gg_e(A) \ar[r]^{\quad s_e} & \hh \ar[r] & (0),
}
    \end{eqnarray}
where $\mathfrak{k}:=\gg_e(A)\cap\kk=\ker(\rho_e\rest{\kk})$.  
\end{lemma}
\begin{proof}
    Since $u$ is an anchored map, $\rho_e(\kk)=\rho_e(A_e)$ and the long exact sequence induced by the map of short exact sequences 
\begin{eqnarray}\label{Eq:mapOfCx}
\xymatrix{
(0) \ar[r] & \kk \ar[r]\ar[d]_{\rho_e} & A_e \ar[r]^{s_e}\ar[d]^{\rho_e} & \hh \ar[r]\ar[d] & (0) \\
(0) \ar[r] & \gg \ar@{=}[r] & \gg \ar[r] & (0) \ar[r] & (0),
}
\end{eqnarray}
is precisely~\eqref{Eq:IsotrpyAlgExt}.
\end{proof}
Now, given $(X,Y)\in T_gG\oplus T_hG$, we write their multiplication as $X\ast Y:=d_{(g,h)}m_G(X,Y)=d_gR_h(X)+d_hL_g(Y)\in T_{gh}G$. That the multiplication is anchored means that the anchor is a Lie groupoid homomorphism, i.e. For each pair $(a_0,a_1)\in A_g\times_\hh A_h$,
\begin{align*}
     \rho_{gh}(a_0\Join a_1) & =\rho_g(a_0)\ast\rho_h(a_1)=d_gR_h(\rho_g(a_0))+d_hL_g(\rho_h(a_1))\in T_{gh}G.
\end{align*}
Let $\gg$ be the Lie algebra of $G$. Observe, that if $x^G\in\XX^R(G)\leq\XX(G)$ is the right-invariant vector field associated to $x\in\gg$, then, for $z\in\kk$ and $g\in G$,
\begin{align}\label{Eq:RtoR}
    \rho_g(z^R_g) & =  \rho_g(z\Join 0_g)=\rho_e(z)\ast 0_g=\rho_e(z)^G_g\in T_gG.
\end{align}
Using the isomorphism induced by $\sigma$ and the right-trivialization of $TG\cong G\times\gg$, one can rewrite the anchor $\rho:A\longrightarrow TG$ as
\begin{eqnarray*}
    \xymatrix{
    \rho:G\times(\kk\oplus\hh) \ar[r] & G\times\gg:(g;z,y) \ar@{|->}[r] & (g;\rho_e(z)+\alpha_g(y)),
    }
\end{eqnarray*}
where $\alpha:G\to\hh^*\otimes\gg$ is given by $\alpha_g(y):=\rho_g(\sigma_g(y))\ast0_{g^{-1}}$ for $g\in G$ and $y\in\hh$.
Note that $\alpha_e\equiv0$. Since for all $g\in G$, $\ker(s_g)\leq A_g$ corresponds to $\lbrace g\rbrace\times\kk\leq \lbrace g\rbrace\times(\kk\oplus\hh)$ under the isomorphism, one concludes that $\rho_g\big{(}\ker(s_g)\big{)}\cong\rho_e(\kk)$. In fact, under the isomorphism, $\rho_g(A_g)$, the image of the anchor map at $g$ is isomorphic to $\rho_e(\kk)+\alpha_g(\hh)$.

Before we move on to the compatibility with the brackets, the following result records the joint behavior of the RUTH $(\partial,\Delta^\hh,\Delta^\kk,\Omega)$ and the anchor map induced by a unit extending splitting $\sigma$. 

\begin{prop}\label{Prop:Rho&RUTH}
    Let $A\rightrightarrows\hh$ be an anchored VB-group over $G$ with core $\kk$ and let $(\partial,\Delta^\hh,\Delta^\kk,\Omega)$ be the RUTH induced by a unit extending splitting $\sigma$ of~\eqref{Eq:CoreSeq}. Then, the anchor is a Lie groupoid homomorphism if and only if $\alpha_e$ vanishes identically, $\rho_g(z\Join0_g)=\rho_e(z)\ast0_g$, and 
    \begin{align}\label{Eq:AnchorRUTH}
        \rho_e(\Delta^\kk_gz) & =\Ad_g(\rho_e(z))+\alpha_g(\partial z), \\ \label{Eq:AnchorRUTH'}
        \alpha_g(\Delta^\hh_hy) & =\alpha_{gh}(y)-\Ad_g(\alpha_h(y))-\rho_e(\Omega_{(g,h)}(y))
    \end{align}
    for all $g,h\in G$, $y\in\hh$ and $z\in\kk$. In particular, when endowed respectively with the representation of Lemma~\ref{lemma:Dela1Indep} and the adjoint representation, $\rho_e$ restricts to a $G$-equivariant map $\ker(\partial)\longrightarrow\gg$.
    \end{prop}
\begin{proof}
    If the anchor is a Lie groupoid homomorphism, given $g\in G$ and $z\in\kk$,
    \begin{align*}
        \rho_e(\Delta^\kk_gz) & =\rho_e(\sigma_g(\partial z)\Join z\Join 0_{g^{-1}})=\rho_g(\sigma_g(\partial z))\ast\rho_{g^{-1}}(z\Join 0_{g^{-1}}) \\
                    & =\alpha_g(\partial z)+ 0_g\ast\rho_e(z)\ast 0_{g^{-1}}=\alpha_g(\partial z)+ \Ad_g(\rho_e(z)).
    \end{align*}
    On the other hand, given $g,h\in G$ and $y\in\kk$,
    \begin{align*}
        \rho_e(\Omega_{(g,h)}(z)) & =\rho_{gh}\big{(}\sigma_{gh}(y)-\sigma_g(\Delta^\hh_hy)\Join\sigma_h(y)\big{)}\ast 0_{(gh)^{-1}} \\
                    & =\alpha_{gh}(y)-\rho_g(\sigma_g(\Delta^\hh_hy))\ast(\alpha_h(y)\ast 0_{g^{-1}}) \\
                    & =\alpha_{gh}(y)-\alpha_g(\Delta^\hh_hy)-\Ad_g(\alpha_h(y)).
        \end{align*}
     Conversely, assuming Eq.'s~\eqref{Eq:AnchorRUTH} and~\eqref{Eq:AnchorRUTH'}, the image of Eq.~\eqref{Eq:StrMult} under the anchor,
     \begin{align*}
         \big(g;\rho_{e}(z_0+\Delta^\kk_gz_1-\Omega_{(g,h)}y)+\alpha_{gh}(y)\big), 
     \end{align*}
      is computed to be
     \begin{align*}
     \big(gh;\rho_e(z_0)+\alpha_g(\partial z_1)+\Ad_g(\rho_e(z_1))+\Ad_g(\alpha_h(y))+\alpha_g(\Delta^\hh_hy)\big);
     \end{align*}
     which, in order, equals $\rho(g;z_0,\Delta_h^\hh y+\partial z_1)\cdot\rho(h;z_1,y)$ in $G_{\Ad}\ltimes\gg\cong TG$. 
\end{proof}


\subsection{Compatibility with the brackets}\label{ssec-brackets}
In this subsection, we study the consequences of the structural maps being Lie algebroid morphisms. In the sequel, we write $[\cdot,\cdot]^1$ for the bracket on $\Gamma(A)$, and $[\cdot,\cdot]^0$ for the bracket on $\hh$. We proceed now to study the compatibility of each of the structural maps with the brackets.


\subsubsection{The unit}\label{sss-unit}

The unit map $u:\hh\longrightarrow A$ is a bundle map that covers the identity element map $\bar{e}:\ast\longrightarrow G$. 
Since $u$ is a Lie algebroid morphism, $u^!:\hh\longrightarrow\Gamma(\bar{e}^!A)$ is a Lie algebra homomorphism. By definition, the pull-back algebroid $\bar{e}^!A$ is simply the isotropy Lie algebra $\gg_e(A)$; hence, the following result follows. 
\begin{lemma}\label{lemma:uBracket}
    If $A\rightrightarrows\hh$ is an LA-group, $u(\hh)$ is a Lie sub-algebra of $\gg_e(A)$.
\end{lemma}


\subsubsection{The source}\label{sss-source}

The source map $s:A\longrightarrow\hh$ is a bundle map that covers the constant map $\bar\ast:G\longrightarrow\ast$. 
Since $s$ is a Lie algebroid morphism, $s^!:\Gamma(A)\longrightarrow\Gamma(\bar\ast^!\hh)$ is a Lie algebra homomorphism. By definition, if $\bar\ast^!\hh:=TG\times\hh$ is the pull-back algebroid, its anchor is given by the projection onto $TG$ and its Lie bracket is given as follows. Since $\bar\ast^!\hh\cong G\times(\gg\oplus\hh)$, the bracket on $\Gamma(\bar\ast^!\hh)=\XX(G)\oplus C(G,\hh)\cong C(G,\gg\oplus\hh)$ is determined by the bracket of constant sections; in order, this is determined by the bracket of the direct sum Lie algebra $\gg\oplus\hh$.

At the level of sections, a unit extending splitting $\sigma$ induces an isomorphism $\Gamma(A)\cong C(G,\kk\oplus\hh)$; in particular, $[\cdot,\cdot]^1$ is also generated by the brackets of \textbf{constant sections}. Given $y\in\hh$, there is a section $y^\sigma\in\Gamma(A)$ that we call the \textbf{constant section associated with} $y$ and is defined by $y^\sigma_g:=\sigma_g(y)$ for $g\in G$. On the other hand, given $z\in\kk$, there is a \textbf{right-invariant section} $z^R\in\Gamma^R(\ker(s))\leq\Gamma(A)$ defined by
\begin{eqnarray}\label{Eq:RInv}
  z^R_g:=z\Join 0_g\in\ker(s_g),
\end{eqnarray}
for $g\in G$ that we refer to also as the \textbf{constant section associated with} $z$. 
With this terminology, Eq.~\eqref{Eq:RtoR} can be phrased as saying that the anchor map sends constant core sections to right-invariant vector fields, the \emph{constant} vector fields on $G$. 

\begin{lemma}\label{lemma:Source}
    Let $A\rightrightarrows\hh$ be a VB-group over $G$ with core $\kk$ and let $\sigma$ be a splitting of its core sequence~\eqref{Eq:CoreSeq}. Suppose further that $A$ and $\hh$ are Lie algebroids. Then, the source map $s:A\to\hh$ is a Lie algebroid morphism if and only if for all $y,y_0,y_1\in\hh$ and $z,z_0,z_1\in\kk$, $$s_g\big([y_0^\sigma,y_1^\sigma]^1_g\big)=[y_0,y_1]^0,\quad\forall g\in G,$$ 
    and $[y^\sigma,z^R]^1,[z_0^R,z_1^R]^1\in\Gamma(\ker(s))\leq\Gamma(A)$.
\end{lemma}
\begin{proof}
    Since $\Gamma(A)$ is generated by constant sections, $s$ is a Lie algebroid morphism if and only if the defining equation $s^![a_0,a_1]^1=[s^!a_0,s^!a_1]_{\bar{\ast}^!\hh}$ holds at constant sections. 
    By definition, for all $y\in\hh$ and all $g\in G$, $s_g(y^\sigma_g)=y$; hence, $s_*(y^\sigma)\in C(G,\hh)$ is the constant function $y$.
    Similarly, for all $z\in\kk$ and all $g\in G$, $s_g(z^R_g)=s_g(z\Join0_g)=0$; hence, $s_*(z^R)\in C(G,\hh)$ is the constant function $0$. 
    Computing thus, yields
    \begin{align*}
    s_*\big([y_0^\sigma,y_1^\sigma]^1\big) & =[y_0,y_1]^0, \quad s_*\big([y^\sigma,z^R]^1\big)  =0, \quad  s_*\big([z_0^R,z_1^R]^1\big)  =0.
    \end{align*}
    as claimed.
\end{proof}

Although Lemma~\ref{lemma:Source} ensures that $[y^\sigma,z^R]^1\in\Gamma(\ker(s))$ for all $y\in\hh$ and $z\in\kk$, in general, there is no reason for it to be right-invariant. 
Instead, this section defines a family 
\begin{eqnarray}\label{Eq:ell}
    \ell:G\longrightarrow\hh^*\otimes\kk^*\otimes\kk:g\longmapsto\ell_g,
\end{eqnarray}
of representation-like tensors parametrized by $G$, where $\ell_g:\hh\longrightarrow\kk^*\otimes\kk$ is given by 
\begin{eqnarray*}
    \ell_g^y(z):=[y^\sigma,z^R]^1_g\Join0_{g^{-1}},
\end{eqnarray*}
whenever $y\in\hh$ and $z\in\kk$. Similarly, although Lemma~\ref{lemma:Source} ensures that $s_g([y^\sigma_0,y^\sigma_1])=[y_0,y_1]^0$ for all $y_0,y_1\in\hh$ and $g\in G$, there is no reason for it to be the constant section associated with $[y_0,y_1]^0$. 
Instead, this section helps defines a family of 2-cochains $\omega:G\longrightarrow\wedge^2\hh^*\otimes\kk$,
\begin{eqnarray}\label{Eq:omega}
    \omega_g(y_0,y_1):=\big{(}[y^\sigma_0,y^\sigma_1]^1_g-\sigma_g([y_0,y_1]^0)\big{)}\Join0_{g^{-1}} ,    
\end{eqnarray}
parametrized by $G$. 
\begin{remark}\label{Rmk:omega=0}
    Note that, if one regards $G\times\hh$ as a bundle of Lie algebras, $\omega\equiv 0$ can be interpreted as saying that $\sigma$ induces a Lie algebra homomorphism at the level of sections; however, unless $\rho\rest{\Img(\sigma)}\equiv0$, $\sigma$ cannot possibly be anchored and hence, $\omega\equiv0$ does not imply that $\sigma$ is a Lie algebroid morphism. 
\end{remark}
\noindent Assuming only that $s:A\to\hh$ is a Lie algebroid morphism, does not imply that the bracket of right-invariant sections is right-invariant either. However, since in an LA-group the multiplication is also a Lie algebroid morphism, as we recall in Lemma~\ref{lemma:RInvBra} below, the space of right-invariant sections $\Gamma^R(\ker(s))$ is indeed a Lie subalgebra of $\Gamma(A)$ (see, e.g. \cite{2ndOrdGeom}). This turns $\kk$ into a Lie algebra and, ultimately, allows one to express the bracket of constant functions in $C(G,\kk\oplus\hh)$ using the families of tensors~\eqref{Eq:ell} and~\eqref{Eq:omega}, given $z_0,z_1\in\kk$ and $y_0,y_1\in\hh$, the bracket of constant functions is given by:
\begin{eqnarray}\label{Eq:Bra}
[(z_0,y_0),(z_1,y_1)]=([z_0,z_1]+\ell^{y_0}(z_1)-\ell^{y_1}(z_0)+\omega(y_0,y_1),[y_0,y_1]^0).    
\end{eqnarray}
\begin{remark}\label{Rmk:NotCocycles}
    Although the bracket of constant functions in $C(G,\kk\oplus\hh)$ assumes the \emph{familiar} shape~\eqref{Eq:Bra}, as in general it is \emph{not} true that the fibre $A_g$ of a Lie algebroid $A$ over $g\in G$ is a Lie algebra, $\omega_g$ has no reason to be a Lie algebra cocycle. In fact, in general, $\ell_g$ does not define a representation either.
\end{remark}
Eq.~\eqref{Eq:Bra} furnishes an explicit expression for the bracket and this, in order, allows one to record some properties that complement Section~\ref{ssec-anchors} and~\ref{sss-unit}. Firstly, since the anchor map of a Lie algebroid always induces a Lie algebra homomorphism at the level of sections, this yields relations between $\rho_e$, $\alpha$, $\ell$ and $\omega$. 

\begin{prop}\label{Prop:anchorBrackets}
    Let $A\rightrightarrows\hh$ be an LA-group over $G$ with core $\kk$ and let $\sigma$ be a splitting of its core sequence~\eqref{Eq:CoreSeq}. 
    Then, the restriction of the anchor map $\rho:A\longrightarrow TG$ to the identity $\rho_e:\kk\longrightarrow\gg$ yields a Lie algebra homomorphism, and  
    \begin{align}\label{Eq:AnchorRels}
        \rho_e(\ell_g^yz) & =[\alpha_g(y),\rho_e(z)]-(d_g\alpha)_{\rho_g(z^R_g)}(y), \\
        [\alpha_g(y_0),\alpha_g(y_1)]-\alpha_g\big([y_0,y_1]^0\big) & =\rho_e\big(\omega_g(y_0,y_1)\big)+(d_g\alpha)_{\rho_g(y_1^\sigma)}(y_0)-(d_g\alpha)_{\rho_g(y_0^\sigma)}(y_1) \nonumber
    \end{align}
    for all $g\in G$, $y,y_0,y_1\in\hh$ and $z\in\kk$.
\end{prop}
\begin{proof}
    That $\rho_e$ is a Lie algebra homomorphism follows from~\eqref{Eq:RtoR}. Keeping in mind that $\alpha$ is not constant, the relations in~\eqref{Eq:AnchorRels} follow from computing point-wise the bracket relations of the constant functions $(0,y),(z,0)\in C(G,\kk\oplus\hh)$ and that of the constant functions $(0,y_0),(0,y_1)\in C(G,\kk\oplus\hh)$, respectively.
\end{proof}

\begin{cor}\label{Cor:lsubalg}
The image of the core under the anchor map at the identity defines a Lie subalgebra $\ll:=\rho_e(\kk)\leq\gg$. Moreover, $\F^\kk:=\rho\big(\Gamma(\ker(s))\big)$ is a right-invariant involutive submodule of $\F^A\leq\XX(G)$.
\end{cor}
Next, one can express the converse to Lemma~\ref{lemma:uAnchord}.
\begin{prop}\label{Prop:uMAP}
    Let $A\rightrightarrows\hh$ be a VB-group over $G$ and let $\sigma$ be a unit extending splitting of its core sequence~\eqref{Eq:CoreSeq}. Suppose further that $A$ and $\hh$ carry Lie algebroid structures for which the source map $s$ is a Lie algebroid morphism. Then, the unit map $u:\hh\to A$ is a Lie algebroid morphism if and only if $\alpha$ and $\omega$ vanish identically at the identity.
\end{prop}
\begin{proof}
    $u$ is a Lie algebroid morphism if and only if its image lies in the isotropy and the map $u:\hh\to\gg_e(A)$ is a Lie algebra homomorphism. 
    Recall that given $a_0,a_1\in\gg_e(A)$, $[a_0,a_1]_{g_e(A)}$ is defined by evaluating $[\bar{a}_0,\bar{a}_1]^1$ at $e$, where $\bar{a}_0,\bar{a}_1\in\Gamma(A)$ are any pair of sections that yield respectively $a_0$ and $a_1$ when evaluated at $e$. In particular, for $y_0,y_1\in\hh$, $[u(y_0),u(y_1)]_
    {\gg_e(A)}$, one might as well use $y_0^\sigma$ and $y_1^\sigma$; therefore, 
    $$[u(y_0),u(y_1)]_{\gg_e(A)}=\omega_e(y_0,y_1)+u([y_0,y_1]^0).$$
    In order, $u$ is a Lie algebra homomorphism if and only if $\omega_e\equiv0$. On the other hand, $\Img(u)\subseteq\gg_e(A)$ if and only if $\alpha_e\equiv0$, thus concluding the proof.
\end{proof}

In sight of~\eqref{Eq:unitSplit} and Lemmas~\ref{lemma:uAnchord} and~\ref{lemma:uBracket}, the following result also ensues.
\begin{cor}\label{Cor:IsoAt e}
    Let $A\rightrightarrows\hh$ be an LA-group with core $\kk$. Then, the exact sequence~\eqref{Eq:IsotrpyAlgExt} happens in the category of Lie algebras. In particular, the isotropy Lie algebra $\gg_e(A)$ is isomorphic to the semi-direct product $\mathfrak{k}\oplus\hh$, where $\mathfrak{k}=\kk\cap\gg_e(A)$.
\end{cor}
\begin{proof}
    It follows from Proposition~\ref{Prop:anchorBrackets} that $\mathfrak{k}$ is an ideal in $\kk$; in particular, it is a subalgebra of $\gg_e(A)$. It is clear from Eq.~\eqref{Eq:Bra} that $s_e$ is a Lie algebra homomorphism, the rest is Lemma~\ref{lemma:uBracket}. 
\end{proof}


\subsubsection{The target}\label{sss-target}

The target map $t:A\longrightarrow\hh$ is a bundle map that covers the constant map $\bar\ast:G\longrightarrow\ast$. 
Since $t$ is a Lie algebroid morphism, $t^!:\Gamma(A)\longrightarrow\Gamma(\bar\ast^!\hh)$ is a Lie algebra homomorphism. By definition, the pull-back algebroid $\bar\ast^!\hh:=TG\times\hh$ is the same as that in the source case, the trivial transitive Lie algebroid with isotropy Lie algebra $\hh$. 
\begin{prop}\label{Prop:Target}
    Let $A\rightrightarrows\hh$ be a VB-group over $G$ with core $\kk$ and let $(\partial,\Delta^\hh,\Delta^\kk,\Omega)$ be the RUTH induced by a unit extending splitting $\sigma$ of~\eqref{Eq:CoreSeq}. Suppose further that $A$ and $\hh$ carry Lie algebroid structures for which the source map $s$ is a Lie algebroid morphism. Then, the target map $t:A\to\hh$ is a Lie algebroid morphism if and only if the structural map $\partial:\kk\longrightarrow\hh$ is a Lie algebra homomorphism, and  
    \begin{align}\label{Eq:TargetRels}
        \partial(\ell_g^yz) & =[\Delta^\hh_gy,\partial z]^0-(d_g\Delta^\hh)_{\rho_e(z)^G_g}(y), \\
        [\Delta^\hh_g(y_0),\Delta^\hh_g(y_1)]^0-\Delta^\hh_g[y_0,y_1]^0 & =\partial\omega_g(y_0,y_1)+(d_g\Delta^\hh)_{\rho_g(y_1^\sigma)}(y_0)-(d_g\Delta^\hh)_{\rho_g(y_0^\sigma)}(y_1) \nonumber
    \end{align}
    for all $g\in G$, $y,y_0,y_1\in\hh$ and $z\in\kk$.
\end{prop}
\begin{proof}
    It follows from~\eqref{Eq:RtoR} that $t_*[z_0^R,z_1^R]=\partial[z_0,z_1]$. Computing the point-wise bracket implies that $\partial=t_e$ restricts to a Lie algebra homomorphism. Keeping in mind that $\Delta^\hh$ is not constant, the relations in~\eqref{Eq:TargetRels} follow from computing point-wise the bracket relations of the constant functions $(0,y),(z,0)\in C(G,\kk\oplus\hh)$ and that of the constant functions $(0,y_0),(0,y_1)\in C(G,\kk\oplus\hh)$, respectively.
\end{proof}
\begin{remark}\label{Rmk:CoreDiagram}
    The seeming similarity between Propositions~\ref{Prop:anchorBrackets} and~\ref{Prop:Target} is not an accident but a feature. Indeed, in a double Lie group, both side groups are subject to the same type of relations. See also \cite{CoreDiagrams}.
\end{remark}
The following result is analogous to Corollaries~\ref{Cor:lsubalg} and~\ref{Cor:IsoAt e}.
\begin{cor}\label{Cor:GpdOrbitsAndIsos}
    Let $A\rightrightarrows\hh$ be an LA-group over $G$. Then, the orbit through $0\in\hh$ defines a Lie subalgebra $O^A_0\leq\hh$. Furthermore, the isotropy Lie group $\G_0(A)$ also carries the structure of a trivial bundle of Lie algebras over $G$. 
\end{cor}
\begin{proof}
    It follows from Proposition~\ref{Prop:GpdOrbits} that $O^A_0=\Img(\partial)$; hence, Proposition~\ref{Prop:Target} implies it is a Lie subalgebra. On the other hand, $G^{\Delta^0}_{\Img(\partial)}=G$; therefore, the zero section splits the sequence~\eqref{Eq:IsoCoreSeq} of Proposition~\ref{Prop:GpdIsotropies} and the fibres are Lie algebras by Proposition~\ref{Prop:Target}.
\end{proof}


\subsubsection{The multiplication}\label{sss-mult}
The multiplication map $m:A\times_\hh A\longrightarrow A$ is a bundle map that covers the multiplication of $G$, $m_G:G\times G\longrightarrow G$. 
Since $m$ is a Lie algebroid morphism, $m^!:\Gamma(A^{(2)})\longrightarrow\Gamma(m_G^!A)$ is a Lie algebra homomorphism. By definition, if $m_G^!\hh:=TG^2\times_{m_G^*TG}m_G^*A$ is the pull-back algebroid, its anchor is given by the projection onto $TG^2$ and its Lie bracket is given as follows. Since $m_G^!A\cong pr_2^*TG\oplus m_G^*A\cong G^2\times(\gg\oplus\kk\oplus\hh)$, the bracket on $\Gamma(m_G^!\hh)=C(G^2)\otimes_{C(G)}\big(\XX(G)\oplus\Gamma(A)\big)\cong C(G^2,\gg\oplus\kk\oplus\hh)$ is determined by the bracket of constant sections; in order, this is given by the direct sum of the Lie brackets in $\gg$ and that of Eq.~\eqref{Eq:Bra}. Now, since $s$ and $t$ are Lie algebroid morphisms, $\Gamma(A^{(2)})\leq\Gamma(A^2)$ and the Lie bracket on $\Gamma(A^{(2)})\cong C( G^2,\kk^2\oplus\hh)$ is easily computed using~\eqref{Eq:Bra} at the level of constant functions. Before writing down the Lie algebra homomorphism condition for $m^!$ at constant sections, recall that the formula~\eqref{Eq:Bra} hinged on $\Gamma^R(\ker(s))$ being a Lie subalgebra of $\Gamma(A)$. That is the content of the next Lemma. 
\begin{lemma}\label{lemma:RInvBra}
    Let $A\rightrightarrows\hh$ be an LA-group with core $\kk$. Then, for all $z_0,z_1\in\kk$, $[z_0^R,z_1^R]^1\in\Gamma(A)$ is right-invariant. 
\end{lemma}
\begin{proof}
    A section $Z\in\Gamma(\ker(s))$ is right-invariant if and only if $m_*(Z,0)=1\otimes Z\in C(G^2)\otimes\Gamma(A)\cong\Gamma(m^*A)$. If $z_0,z_1\in\kk$, it follows from Lemma~\ref{lemma:Source} that $[z_0^R,z_1^R]^1\in\ker(s)$. Thus, since the anchor map is always a Lie algebra homomorphism, the first entry in the relation $m^!\big([z_0^R,z_1^R]^1,0\big)=\big[m^!(z_0^R,0),m^!(z_1^R,0)\big]$ is inconsequential. The second entry, on the other hand, reads $$m_*\big([z_0^R,z_1^R]^1,0\big)=1\otimes[z_0^R,z_1^R]^1,$$ as desired.    
\end{proof}
We collect all the remaining bracket relations in the following Proposition.
\begin{prop}\label{Prop:AllMult}
    Let $A\rightrightarrows\hh$ be a VB-group over $G$ with core $\kk$ and let $(\partial,\Delta^\hh,\Delta^\kk,\Omega)$ be the RUTH induced by a unit extending splitting $\sigma$ of~\eqref{Eq:CoreSeq}. Suppose further that $A$ and $\hh$ carry Lie algebroid structures for which the source and the target maps $s$ and $t$ are Lie algebroid morphisms. Then, the multiplication $m:A^{(2)}\to\hh$ is a Lie algebroid morphism if and only if the anchor $\rho$ is a Lie groupoid homomorphism, the bracket of right-invariant sections is right-invariant, and the following relations hold 
    for all $g,h\in G$, $z_0,z_1,z\in\kk$ and $y,y_0,y_1\in\hh$:
    \begin{align*}
        \textnormal{(M01) } & 
        \ell_g^{\partial z_0}(z_1)=[\Delta^\kk_g z_0,z_1]-(d_g\Delta^\kk)_{\rho_e(z_1)^G_g}(z_0), \\ 
        \textnormal{(M02) } & 
        \ell_{gh}^y(z)-\ell_{g}^{\Delta^\hh_h y}(z)=[\Omega_{(g,h)}(y),z]-(d_{(g,h)}\Omega)_{(\rho_e(z),0)^{G^2}_{(g,h)}}(y), \\
        \textnormal{(M11) } & \omega_g(\partial z_0,\partial z_1)=
        [\Delta^\kk_gz_0,\Delta^\kk_gz_1]-\Delta^\kk_g[z_0,z_1]+ \\
         & \qquad\qquad\qquad\qquad\qquad\quad+(d_g\Delta^\kk)_{\rho_g(\partial z_0^\sigma)}(z_1)-(d_g\Delta^\kk)_{\rho_g(\partial z_1^\sigma)}(z_0), \\
        \textnormal{(M12) } & \Delta^\kk_g\ell_h^y(z)-\ell_{gh}^y(\Delta^\kk_gz)+\omega_g(\Delta^\hh_h y,\partial z)=[\Omega_{(g,h)}(y),\Delta^\kk_gz]+ \\
            & \qquad\qquad\qquad\qquad +(d_g\Delta^\kk)_{\rho_g(\Delta^\hh_hy^\sigma)}(z)-(d_{(g,h)}\Omega)_{(\rho_g(\partial z^\sigma),\rho_e(z)^G_h)}(y),
    \\
         \textnormal{(M22) } & 
        \Delta^\kk_g\omega_h(y_0,y_1)-\omega_{gh}(y_0,y_1)+\omega_g(\Delta^\hh_hy_0,\Delta^\hh_hy_1)+ \\
            +\ell_{gh}^{y_0}( & \Omega_{(g,h)}(y_1))-\ell_{gh}^{y_1}(\Omega_{(g,h)}(y_0))=\Omega_{(g,h)}([y_0,y_1]^0)+[\Omega_{(g,h)}(y_0),\Omega_{(g,h)}(y_1)]+ \\
            &  +(d_{(g,h)}\Omega)_{\big(\rho_g(\Delta^\hh_hy_1^\sigma),\rho_h(y_1^\sigma)\big)}(y_0)-(d_{(g,h)}\Omega)_{\big(\rho_g(\Delta^\hh_hy_0^\sigma),\rho_h(y_0^\sigma)\big)}(y_1).   
    \end{align*}
\end{prop}
\begin{proof}
    Given constant $(z_0,z_1,y),(z',0,0)\in C(G^2,\kk^2\oplus\hh)$, 
    \begin{eqnarray*}\label{Eq:adZ0}
        \big[(z_0,z_1,y),(z',0,0)\big]  =\Big([z_0,z']+pr_1^*\ell^{\partial z_1+pr_2^*\Delta^\hh y}(z'),0,0\Big).
    \end{eqnarray*}
    The relations (M01) and (M02) correspond to the independent components of the bracket relation $m^!
        \big[(0,z_0,y),(z_1,0,0)\big]=
        \big[m^!(0,z_0,y),m^!(z_1,0,0)\big]$.
    Similarly, given constant $(0,z_0,y),(0,z_1,0)\in C(G^2,\kk^2\oplus\hh)$, 
    \begin{eqnarray*}\label{Eq:adZ1}
    \big[(0,z_0,y),(0,z_1,0)\big]  =\Big(pr_1^*\big(\omega(\partial z_0+pr_2^*\Delta^\hh y,\partial z_1)
    \big),[z_0,z_1]+pr_2^*\ell^y(z_1),0\Big), 
    \end{eqnarray*}
    and the relations (M11) and (M12) correspond to the independent components of the bracket relation $m^!
        \big[(0,z_0,y),(0,z_1,0)\big]=
        \big[m^!(0,z_0,y),m^!(0,z_1,0)\big]$.
    Lastly, given constant $(0,0,y_0),(0,0,y_1)\in C(G^2,\kk^2\oplus\hh)$, 
    \begin{eqnarray*}\label{Eq:adY}
    \big[(0,0,y_0),(0,0,y_1)\big]  =\big(pr_1^*\omega(pr_2^*\Delta^\hh y_0,pr_2^*\Delta^\hh y_1),pr_2^*\omega(y_0,y_1),[y_0,y_1]^0\big),
    \end{eqnarray*}
    and the relation (M22) corresponds to the bracket relation $m^!
        \big[(0,0,y_0),(0,0,y_1)\big]=
        \big[m^!(0,0,y_0),m^!(0,0,y_1)\big]$.   
\end{proof}


\section{Applications to structure}\label{sec:4}
In this section we collect some consequences of the results of Section~\ref{sec:3}. 

\noindent We divide these results in two parts. 
In the first part, mimicking \cite[Theorem 3.6]{MackJotzRaj}, 
we extend Theorem~\ref{Theo:MainEqce} to include the bracket and the anchor in the structure (see Theorem~\ref{Theo:MainEqceSuite}). 
In the second part, we derive results about the structure of an LA-group that parallel the results in Subsection~\ref{ssec:2}.


\subsection{LA-Matched Pairs}\label{ssec-LAMtchPrs}
LA-groupoids, in general, can be seen as an intermediary step between double Lie groupoids and double Lie algebroids.
Both these objects are more symmetric than LA-groupoids in that both of their \emph{side structures} are the same. 
In the case of double Lie algebroids, each side can be regarded as VB-algebroid and, consequently \cite{DoubleVB}, each side corresponds to a representation up to homotopy. Since a similar thing happens in the case of double Lie groupoids \cite{CoreDiagrams}, it stands to reason that, in the case of LA-groupoids, there is some structure complemeting the RUTH $(\partial,\Delta^\hh,\Delta^\kk,\Omega)$ of $G$ on $\kk[1]\oplus\hh$ that defines the VB-groupoid $G\times(\kk\oplus\hh)\rightrightarrows\hh$ over $G$ via Theorem~\ref{Theo:MainEqce}. In this subsection, restricting again to LA-groups, we define this complementary structure as a \textbf{Lie algebra action up to homotopy} (see Definition~\ref{Def:AUTH}). Moreover, in analogy with \cite{MackJotzRaj}, we show that in order for the VB-group $G\times(\kk\oplus\hh)\rightrightarrows\hh$ to turn into an LA-group, one needs an action up to homotopy of $\hh$ on $G\times\kk$ that interacts with the RUTH in a particular manner (see Theorem~\ref{Theo:MainEqceSuite}).

\begin{definition}\label{Def:AUTH}
    An \textbf{action up to homotopy} (AUTH for short) of a Lie algebra $\hh$ on a Lie algebroid $E\to M$ with anchor $\rho:E\longrightarrow TM$ is a triple $(\alpha,\nabla,\omega)$ consisting of linear maps $\alpha:\hh\longrightarrow\XX(M)$, $\nabla:\hh\otimes_\Rr\Gamma(E)\longrightarrow\Gamma(E)$, and a smooth section $\omega\in\Gamma((G\times\wedge^2\hh^*)\otimes E)$ further subject to the following relations:   
    \begin{itemize}
        \item[(i)] $\nabla$ is a connection with respect to $\alpha$ and it acts by derivations; in symbols,
        \begin{align}\label{Eq:AUTHi-connection}
            \nabla_y(\phi\epsilon) & =\phi\nabla_y(\epsilon)+(\Lie_{\alpha(y)})\epsilon, \\ \label{Eq:AUTHi-byDerv}
            \nabla_y([\epsilon_0,\epsilon_1]_E) & =[\nabla_y(\epsilon_0),\epsilon_1]_E+[\epsilon_0,\nabla_y(\epsilon_1)]_E,
        \end{align}
        for all $y\in\hh$, $\phi\in C^\infty(M)$, and $\epsilon,\epsilon_0,\epsilon_1\in\Gamma(E)$.
        \item[(ii)] The anchor map $\rho$ is equivariant; in symbols,
        \begin{align}\label{Eq:AUTHii-Equivce}
            \rho(\nabla_y(\epsilon)) & =[\alpha(y),\rho(\epsilon)]_{\XX(M)}, 
        \end{align}
        for all $y\in\hh$, and $\epsilon\in\Gamma(E)$.
        \item[(iii)] If $R^\nabla$ is the curvature of the connection $\nabla$, 
        \begin{align}\label{Eq:AUTHiii-CurvControl}
            R^\nabla_{(y_0,y_1)}(\epsilon) & =[\omega(y_0,y_1),\epsilon]_E, 
        \end{align}
        for all $y_0,y_1\in\hh$, and $\epsilon\in\Gamma(E)$.
        \item[(iv)] As a 2-cochain with values in $\Gamma(E)$, $\omega\in\wedge^2\hh^*\otimes\Gamma(E)$ is \emph{formally} a 2-cocycle, i.e. $\d_\nabla\omega\equiv0\in\wedge^3\hh^*\otimes\Gamma(E)$.
    \end{itemize}
\end{definition}
\begin{remark}\label{Rmk:AUTH}
    In the context of Definition~\ref{Def:AUTH}, given $y_0,y_1\in\hh$, $\phi\in C^\infty(M)$, and $\epsilon\in\Gamma(E)$, one can easily compute 
    $$R^\nabla_{(y_0,y_1)}(\phi\epsilon)=\phi R^\nabla_{(y_0,y_1)}(\epsilon)+(\Lie_{R^\alpha(y_0,y_1)}\phi)\epsilon,$$
    where $R^\alpha(y_0,y_1)=[\alpha(y_0),\alpha(y_1)]_{\XX(M)}-\alpha([y_0,y_1]^0)$ is the curvature of the \emph{quasi-action} $\alpha$. 
    It follows now from Eq.~\eqref{Eq:AUTHiii-CurvControl}, that the left-hand side equals $[\omega(y_0,y_1),\phi\epsilon]_E$, 
    which, by the Leibnitz identity, implies $R^\alpha(y_0,y_1)=\rho(\omega(y_0,y_1))$, thus justifying calling it an \emph{action} up to homotopy. 
    
    \noindent This notion can presumably be recast in the language of graded manifolds (cf. \cite{DoubleVB}). 
    The reader is also invited to compare the equations defining an AUTH with those defining a RUTH of a Lie algebroid (cf. \cite[Remark 3.7.]{AAC}). 
\end{remark}
The AUTHs of Definition~\ref{Def:AUTH} correspond to \emph{extensions} of Lie algebroids together with a splitting in the sense that RUTHs correspond to VB-groupoids together with a splitting of their core sequences. The following result explains this fact more precisely.
\begin{prop}\label{Prop:MainAUTH}
    Let $(\alpha,\nabla,\omega)$ be an AUTH of the Lie algebra $\hh$ on the Lie algebroid $\pi:E\to M$ with anchor $\rho:E\longrightarrow TM$. Then, there is a Lie algebroid $\mathcal{A}:=E\times\hh$ whose anchor $\rho_\mathcal{A}:\mathcal{A}\to TM$ is given by $\rho_\mathcal{A}(e,y):=\rho(e)+\alpha_{\pi(e)}(y)$ and whose bracket is given by extending the bracket of $\hh$-constant sections $(\epsilon_0,y_0),(\epsilon_1,y_1)\in\Gamma(E)\times\hh$, 
    \begin{align*}
    [(\epsilon_0,y_0),(\epsilon_1,y_1)]_\mathcal{A} & =([\epsilon_0,\epsilon_1]_E+\nabla_{y_0}(\epsilon_1)-\nabla_{y_1}(\epsilon_0)+\omega(y_0,y_1),[y_0,y_1]^0)    
    \end{align*}
    using the Leibnitz identity. Moreover, there is a short exact sequence in the category of Lie algebroids
    \begin{eqnarray*}
    \xymatrix{
        (0) \ar[r] & E \quad\ar@{^{(}->}[r] & \mathcal{A} \ar[r]^{\textnormal{pr}_\hh} & \hh \ar[r] & (0).
    }
    \end{eqnarray*}
    Conversely, given a surjective Lie algebroid morphism $F:\mathcal{A}\longrightarrow\hh$, letting $E=\ker(F)$, there is a short exact sequence in the category of vector bundles 
    \begin{eqnarray}\label{Eq:AUTHExtVB}
    \xymatrix{
        (0) \ar[r] & E\quad \ar@{^{(}->}[r] & \mathcal{A} \ar[r]^{F_\ast\quad} & M\times\hh \ar[r] & (0),
    }
    \end{eqnarray}
    and any splitting $\sigma$ of \eqref{Eq:AUTHExtVB} yields an AUTH of $\hh$ on $E$ given by 
    \begin{align*}
        \alpha:\hh\longrightarrow\XX(M),\quad \alpha(y) & :=\rho_\mathcal{A}(y^\sigma)\textnormal{ for }y\in\hh, \\
        \nabla:\hh\otimes_\Rr\Gamma(E)\longrightarrow \Gamma(E),\quad \nabla_y(\epsilon) & :=[y^\sigma,\epsilon]_{\mathcal{A}}\textnormal{ for }y\in\hh,\epsilon\in\Gamma(E), \\
        \omega:\wedge^2\hh\longrightarrow\Gamma(E),\quad \omega(y_0,y_1) & :=[y_0^\sigma,y_1^\sigma]_{\mathcal{A}}-([y_0,y_1]^0)^\sigma\textnormal{ for }y_0,y_1\in\hh.
    \end{align*}
\end{prop}
In the upcoming example, we are a bit careless with the overlapping notation between the quasi-action $\alpha$ of an AUTH and the $\alpha$ of Section~\ref{sec:3}. 
Though they are not equal, they are related by the right trivialization of $TG\cong G\times\gg$.
In the sequel, we abuse notation a bit and define vector fields associated to constant functions in $C(G,\kk\oplus\hh)$ as follows: 
For $z\in\kk$ and $y\in\hh$, $\rho(z),\alpha(y)\in\XX(G)$ are given by $\rho(z)_g:=\rho_e(z)\ast0_g$ and $\alpha(y)_g:=\alpha_g(y)\ast 0_g$ for $g\in G$.
\begin{ex:}\label{Ex:LA-AUTH}
    Let $A\rightrightarrows\hh$ be an LA-group over $G$ and let $\sigma$ be a splitting of its core sequence~\eqref{Eq:CoreSeq}. Then, since $s:A\to\hh$ is a Lie algebroid morphism, Proposition~\ref{Prop:MainAUTH} implies there is an AUTH $(\alpha,\nabla,\omega)$ of $\hh$ on the subalgebroid $\ker(s)\leq A$, where $\alpha$ and $\omega$ coincide with those in Section~\ref{sec:3}, and where, using the isomorphism $\Gamma(\ker(s))\cong C(G,\kk)$,
    \begin{align}\label{Eq:nabla}
        (\nabla_yZ)_g=\ell_g^y(Z_g)+(d_gZ)_{\alpha(y)_g}
    \end{align}
    for $y\in\hh$, $g\in G$ and $Z\in C(G,\kk)$. Here, $\ell$ is also that of Section~\ref{sec:3} (see~\eqref{Eq:ell}).

    \noindent Note that if $\sigma$ extends the unit, both $\alpha_e$ and $\omega_e$ vanish identically. 
\end{ex:}
The AUTH of Example~\ref{Ex:LA-AUTH} is, of course, very special. 
Firstly, if $\kk$ is the core of $A\rightrightarrows\hh$, the Lie algebroid $\ker(s)$ is isomorphic to $G\times\kk$, whose anchor and bracket are determined by their values at the identity. Furthermore, if $\sigma$ extends the unit, Theorem~\ref{Theo:MainEqce} implies there is a RUTH $(\partial,\Delta^\hh,\Delta^\kk,\Omega)$ of $G$ on $\kk[1]\oplus\hh$, and
as it turns out the Lie bracket on $\kk$ is constrained by Eq.~(M01) in Proposition~\ref{Prop:AllMult} at $e$ to be
\begin{eqnarray}\label{Eq:coreBra}
    [z_0,z_1]=\ell_{e}^{\partial z_0}(z_1)+(d_{e}\Delta^\kk)_{\rho_e(z_1)}(z_0).
\end{eqnarray}
That is, the Lie algebroid structure on $\ker(s)$ is determined by $\rho_e$, $\ell_e$, $\partial$ and $\Delta^\kk$.
Similarly, the connection $\nabla$ is also ultimately determined by the RUTH, $\alpha$ and its value at constant functions at the identity; 
indeed, $\nabla$ is determined by $\alpha$ and $\ell$ (see Eq.~\eqref{Eq:nabla}), 
and evaluating Eq.~(M02) in Proposition~\ref{Prop:AllMult} at $(e,g)$ yields 
\begin{eqnarray}\label{Eq:ellsuite}
    \ell_{g}^y(z)=\ell_{e}^{\Delta^\hh_g y}(z)-(d_{(e,g)}\Omega)_{(\rho_e(z),0_g)}(y).
\end{eqnarray}
Hence, $\ell$ is defined in terms of $\rho_e$, $\ell_e$, $\Delta^\hh$ and $\Omega$. 
Out of the AUTH of Example~\ref{Ex:LA-AUTH}, one gets a 4-tuple $(\rho_e,\alpha,\ell_e,\omega)$, 
where, after trivializing by right translations, 
$\alpha\in C(G,\hh^*\otimes\gg)$ and $\omega\in C(G,\wedge^2\hh^*\otimes\kk)$ are functions that vanish at the identity. 
This 4-tuple resembles more clearly a representation up to homotopy, 
as the defining equations for the AUTH restrict to   
\begin{align}\label{Eq:AnchorRel1@e}
    \rho_e(\ell_e^yz) & =-(d_e\alpha)_{\rho_e(z)}(y), \\ \label{Eq:AnchorRel2}
    [\alpha(y_0),\alpha(y_1)]-\alpha([y_0,y_1]^0) & =\rho_e(\omega(y_0,y_1)), \\ \label{Eq:ellRep@e}
    \big([\ell_e^{y_0},\ell_e^{y_1}]-\ell_e^{[y_0,y_1]^0}\big)(z) & =(d_e\omega)_{\rho_e(z)}(y_0,y_1),
\end{align}
plus the cocycle equation for $\omega$.

\noindent Before we state the next definition, we introduce the following notation. 
Let $\Delta^\hh$ and $\Delta^\kk$ be quasi-actions of $G$ on $\hh$ and $\kk$, respectively. Then, we define a degree $+1$ map
\begin{align*}
    \delta^\Delta & :C(G^\bullet,\hh^*\otimes\kk)\longrightarrow C(G^{\bullet+1},\hh^*\otimes\kk) 
\end{align*}
$$(\delta^\Delta\theta)_{(g_0,...,g_n)}:=\Delta^\kk_{g_0}\circ\theta_{(g_1,...,g_n)}+\sum_{k=1}^{n}(-1)\theta_{\delta_k(g_0,...,g_n)}+(-1)^{n+1}\theta_{(g_0,...,g_{n-1})}\circ\Delta^\hh_{g_n},$$
where $\theta\in C(G^n,\hh^*\otimes\kk)$, $g_0,...,g_n\in G$, and $\delta_k$ stands for the $k$-th face map in the nerve of $G$. 
Computing $(\delta^\Delta\circ\delta^\Delta)\theta$ and evaluating at $(g_0,...,g_{n+1})\in G^{n+2}$ yields 
$$(\Delta^\kk_{g_0}\circ\Delta^\kk_{g_1}-\Delta^\kk_{g_0g_1})\circ\theta_{(g_2,...,g_{n+1})}+\theta_{(g_0,...,g_{n-1})}\circ(\Delta^\hh_{g_ng_{n+1}}-\Delta^\hh_{g_n}\circ\Delta^\hh_{g_{n+1}});$$
therefore, $\delta^\Delta$ is not a differential unless both $\Delta^\hh$ and $\Delta^\kk$ are representations. 
If the quasi-actions define a RUTH along with some $\partial$ and some $\Omega$, then $\delta^\Delta\Omega=0$. 
This justifies our calling this equation the \emph{cocycle equation for the curvature} in the proof of Proposition~\ref{Prop:GpdIsotropies} and in Remark~\ref{Rmk:WhyNotOmega}. 
We abuse notation a bit and write $\delta^\Delta$ when taking the quasi-action $\Delta^\hh\times\Delta^\hh$ on $\wedge^2\hh$.

\begin{definition}\label{Def:LAMtchdPair}, 
    Let $G$ be a Lie group with Lie algebra $\gg$ and identity element $e$, let $\hh$ be a Lie algebra, and let $\kk$ be a vector space. Further, let $$(\rho_e,\alpha,\ell_e,\omega)\in(\kk^*\otimes\gg)\times C(G,\hh^*\otimes\gg)\times(\hh^*\otimes\kk^*\otimes\kk)\times C(G,\wedge^2\hh^*\otimes\kk),$$
    where both $\alpha$ and $\omega$ vanish identically at the identity.
    A RUTH $(\partial,\Delta^\hh,\Delta^\kk,\Omega)$ of $G$ on $\kk[1]\oplus\hh$ and $(\rho_e,\alpha,\ell_e,\omega)$ are said to form an \textbf{LA-matched pair} if 
    \begin{enumerate}
        \item Eq.~\eqref{Eq:coreBra} defines a Lie bracket on $\kk$ and consequently a Lie algebroid structure on $G\times\kk$ by extending the bracket of constant sections using the anchor $\rho(g;z)=\rho_e(z)\ast0_g$ and the Leibnitz identity;
        \item Eq.~\eqref{Eq:nabla}, where $\ell_g$ is given by Eq.~\eqref{Eq:ellsuite}, defines a connection that turns $(\alpha,\nabla,\omega)$ into an AUTH of $\hh$ on $G\times\kk$;
        \item $\partial$ is $\hh$-equivariant up to $\Img(\rho_e)$, i.e. Given $y\in\hh$ and $z\in\kk$, 
        \begin{align}\label{Eq:h-equiv}
            [y,\partial z]^0-\partial(\ell_e^yz)=(d_e\Delta^\hh)_{\rho_e(z)}(y);
        \end{align}
        \item $\rho_e$ is $G$-equivariant up to $\Img(\partial)$, i.e. Given $g\in G$ and $z\in\kk$, 
        \begin{align}\label{Eq:G-equiv}
            \rho_e(\Delta^\kk_g z)-\Ad_g(\rho_e(z))=\alpha_{g}(\partial z);
        \end{align}
        \item $\Delta^\hh$ and $\alpha$ combine to define a connection on the product Lie algebroid $TG\times\hh$ that is flat up to $\omega$, i.e. The connection 
        \begin{align*}
            \nabla^\hh:\hh\otimes\big(\XX(G)\oplus & C(G,\hh)\big)\longrightarrow\XX(G)\oplus C(G,\hh), \\
            \nabla^\hh_y(X,Y) & :=[(\alpha(y),\Delta^\hh y),(X,Y)]
        \end{align*}
        has curvature 
        \begin{align}\label{Eq:AlmostFlat}
            R^{\nabla^\hh}_{(y_0,y_1)}(X,Y)=\big[(\rho(\omega(y_0,y_1)),\partial\omega(y_0,y_1)),(X,Y)\big]
        \end{align}
        for $y_0,y_1\in\hh$ and $(X,Y)\in\XX(G)\oplus C(G,\hh)$;
        \item $\Delta^\hh$ and $\alpha$ combine to define a quasi-action on $\gg\oplus\hh$ that is a representation up to $\Omega$, i.e. The map
        \begin{align*}
            \Xi^\hh & :G\longrightarrow\textnormal{End}(\gg\oplus\hh), \\
            \Xi_g^\hh(x,y) & :=(\Ad_g(x)+\alpha_g(y),\Delta^\hh_g y)
        \end{align*}
        is such that  
        \begin{align}\label{Eq:AlmostRepn}
        \Xi^\hh_{gh}(x,y)-\Xi^\hh_{g}\circ\Xi^\hh_{h}(x,y)=\big(\rho_e(\Omega_{(g,h)}(y)),\partial\Omega_{(g,h)}(y)\big)
        \end{align}
        for $g,h\in G$ and $(x,y)\in\gg\oplus\hh$;
        \item $\Delta^\kk$ and $\nabla$ commute up to the curvatures, i.e. Given $g\in G$, $y\in\hh$ and $z\in\kk$, 
         \begin{align}\label{Eq:AlmostCommute}       
            (\nabla_y(\Delta^\kk z))_g-\Delta^\kk_g(\nabla_y(z))=\omega_g(y,\partial z)+(d_{(g,e)}\Omega)_{(0_g,\rho_e(z))}(y);
         \end{align}
        \item $\Omega\in\hh^*\otimes C(G^2,\kk)$ verifies a Maurer-Cartan type equation up to the coboundary $\delta^\Delta\omega\in \wedge^2\hh^*\otimes(G^2,\kk)$ with respect to the extended connection 
        \begin{align}\label{Eq:nabla2}
            (\nabla_yZ)_{(g,h)}=\ell_{gh}^y(Z_{(g,h)})+(d_{(g,h)}Z)_{(\alpha(\Delta^\hh_hy)_g,\alpha(y)_h)}
        \end{align}
        for $y\in\hh$, $(g,h)\in G^2$ and $Z\in C(G^2,\kk)$, i.e. Given $g,h\in G$ and $y_0,y_1\in\hh$,
        \begin{align}\label{Eq:MCEq}       
            (\d_\nabla\Omega)_{(g,h)}(y_0,y_1)+[\Omega_{(g,h)}(y_0),\Omega_{(g,h)}(y_1)]=(\delta^\Delta\omega)_{(g,h)}(y_0,y_1).
        \end{align}
    \end{enumerate}
\end{definition}

\begin{remark}\label{Rmk:MatchedPairs}
    Definition~\ref{Def:LAMtchdPair} should be compared with the definition of matched pairs of representations up to homotopy of Lie algebroids \cite[Definition 3.1]{MackJotzRaj}. A RUTH of the Lie group $G$ on $\kk[1]\oplus\hh$ differentiates to a RUTH of its Lie algebra $\gg$; similarly, an AUTH of $\hh$ on $G\times\kk$ differentiates to a RUTH of $\hh$ on $\kk[1]\oplus\gg$. The defining equations of Definition~\ref{Def:LAMtchdPair}, in order, differentiate to Eq.'s \emph{(M1)-(M7)} in \cite[Definition 3.1]{MackJotzRaj}. Specifically, Eq.~\eqref{Eq:coreBra} implies Eq.~\emph{(M1)}; Eq.~\eqref{Eq:h-equiv} implies Eq.~\emph{(M2)}; Eq.~\eqref{Eq:G-equiv} implies Eq.~\emph{(M3)}; Eq.~\eqref{Eq:AlmostCommute} implies Eq.~\emph{(M4)}; Eq.~\eqref{Eq:AlmostFlat} implies Eq.~\emph{(M5)}; Eq.~\eqref{Eq:AlmostRepn} implies Eq.~\emph{(M6)}; and, lastly, Eq.~\eqref{Eq:MCEq} implies Eq.~\emph{(M7)}.    
\end{remark}

Definition~\ref{Def:LAMtchdPair} allows us to state the main result of this section, Theorem~\ref{Theo:MainEqceSuite}. This result extends Theorem~\ref{Theo:MainEqce} to LA-groups and \cite[Theorem 3.6]{MackJotzRaj}.
\begin{theorem}\label{Theo:MainEqceSuite}
Let $A\rightrightarrows\hh$ be an LA-group over $G$ with core $\kk$. Given a unit extending splitting of~\eqref{Eq:CoreSeq}, the induced RUTH $(\partial,\Delta^\hh,\Delta^\kk,\Omega)$ of $G$ on $\kk[1]\oplus\hh$ and the 4-tuple $(\rho_e,\alpha,\ell_e,\omega)$ form an LA-matched pair. Conversely, the VB-group $G\times(\kk\oplus\hh)\rightrightarrows\hh$ constructed using the RUTH $(\partial,\Delta^\hh,\Delta^\kk,\Omega)$ via Theorem~\ref{Theo:MainEqce} is an LA-group if and only one is given a 4-tuple $(\rho_e,\alpha,\ell_e,\omega)$, which together with the RUTH form an LA-matched pair.
\end{theorem}

Since we are interested in LA-groups, where by definition, each of the structural maps is a Lie algebroid morphism, 
the statement of Theorem~\ref{Theo:MainEqceSuite} needs all of the axioms in Definition~\ref{Def:LAMtchdPair}. 
However, one could be more specific and point out which axiom implies that a particular map is a Lie algebroid morphism. 
We postpone this level of specificity to the proof, which will follow from the following technical lemma along with the results of Section~\ref{sec:3}.

\begin{lemma}\label{Lemma:ell}
    If $(\partial,\Delta^\hh,\Delta^\kk,\Omega)$ and $(\rho_e,\alpha,\ell_e,\omega)$ form an LA-matched pair, then Eq.'s~\eqref{Eq:AnchorRels},~\eqref{Eq:TargetRels} as well as Eq.'s~$\textnormal{(M01)}$,~$\textnormal{(M02)}$, and~$\textnormal{(M11)}$ in Proposition~\ref{Prop:AllMult} hold.
\end{lemma}
\begin{proof}
    Let $y\in\hh$, $z,z_0,z_1\in\kk$ and $g,h\in G$ throughout. We start with Eq.~\eqref{Eq:AnchorRels}; 
    using Eq.~\eqref{Eq:ellsuite}, we compute
    \begin{align*}
        \rho_e(\ell_g^yz) & =\rho_e\big(\ell_{e}^{\Delta^\hh_g y}(z)-(d_{(e,g)}\Omega)_{(\rho_e(z),0_g)}(y)\big) \\
                & = -(d_e\alpha)_{\rho_e(z)}(\Delta^\hh_gy)-\rho_e\big((d_{(e,g)}\Omega)_{(\rho_e(z),0_g)}(y)\big).
    \end{align*}
    Here, the second equality follows from Eq.~\eqref{Eq:AnchorRel1@e}. 
    Now, Eq.~\eqref{Eq:AnchorRels} follows from differentiating the $\gg$-coordinate in Eq.~\eqref{Eq:AlmostRepn} and evaluating at $(\rho_e(z_1),0_g)\in\gg\oplus T_gG$.

    \noindent To deduce Eq.~\eqref{Eq:TargetRels}, we use Eq.~\eqref{Eq:ellsuite} again and compute
    \begin{align*}
        \partial(\ell_g^yz) & =\partial\big(\ell_{e}^{\Delta^\hh_g y}(z)-(d_{(e,g)}\Omega)_{(\rho_e(z),0_g)}(y)\big) \\
                & = [\Delta^\hh_gy,z]^0-(d_e\Delta^\hh)_{\rho(z)}(\Delta^\hh_gy)-\rho_e\big((d_{(e,g)}\Omega)_{(\rho_e(z),0_g)}(y)\big).
    \end{align*}
    Here, the second equality follows from Eq.~\eqref{Eq:h-equiv}, 
    and Eq.~\eqref{Eq:TargetRels} follows from differentiating $\Delta^\hh_{gh}-\Delta^\hh_g\circ\Delta^\hh_h=\partial\circ\Omega_{(g,h)}$ and evaluating at $(\rho_e(z_1),0_g)\in\gg\oplus T_gG$.

    \noindent To deduce Eq.~$\textnormal{(M01)}$ in Proposition~\ref{Prop:AllMult}, we use Eq.~\eqref{Eq:ellsuite} and compute
    \begin{align*}
        \ell_g^{\partial z_0}(z_1) & =\ell_e^{\Delta^\hh_g\partial z_0}(z_1)-(d_{(e,g)}\Omega)_{(\rho(z_1),0)}(\partial z_0) \\
            & =[\Delta^\kk_g z_0,z_1]-(d_e\Delta^\kk)_{\rho(z_1)}(z_0)-(d_{(e,g)}\Omega)_{(\rho(z_1),0)}(\partial z_0).
    \end{align*}
    Here, the second equality follows from $\Delta^\hh\circ\partial=\partial\circ\Delta^\kk$ and Eq.~\eqref{Eq:coreBra}. 
    Eq.~$\textnormal{(M01)}$ now follows from differentiating $\Delta^\kk_{gh}-\Delta^\kk_g\circ\Delta^\kk_h=\Omega_{(g,h)}\circ\partial$ and evaluating at $(\rho_e(z_1),0_g)\in\gg\oplus T_gG$.
    
    \noindent To deduce Eq.~$\textnormal{(M02)}$ in Proposition~\ref{Prop:AllMult}, we use Eq.~\eqref{Eq:ellsuite} once more and compute
    \begin{eqnarray*}
    \resizebox{1\hsize}{!}{$ \big(\ell_{gh}^y-\ell_g^{\Delta^\hh_hy}\big)(z)=\big(\ell_{e}^{\Delta^\hh_{gh}y-\Delta^\hh_g\Delta^\hh_hy}\big)(z)-(d_{(e,gh)}\Omega)_{(\rho(z),0)}(y)+(d_{(e,g)}\Omega)_{(\rho(z),0)}(\Delta^\hh_hy).$}
    \end{eqnarray*}
    Since $\Delta^\hh_{gh}-\Delta^\hh_g\Delta^\hh_h=\partial\Omega_{(g,h)}$, we use Eq.~\eqref{Eq:coreBra} and deduce
    \begin{eqnarray*}
        \ell_{e}^{\partial\Omega_{(g,h)}(y)}(z)=[\Omega_{(g,h)}(y),z]-(d_{e}\Delta^\kk)_{\rho(z)}(\Omega_{(g,h)}(y)).
    \end{eqnarray*}
        Eq.~$\textnormal{(M02)}$ now follows from differentiating the cocycle equation $\delta^\Delta\Omega=0$, and evaluating at $(\rho_e(z),0_g,0_h)\in\gg\oplus T_gG\oplus T_hG$. 

    \noindent Lastly, to deduce Eq.~$\textnormal{(M11)}$ in Proposition~\ref{Prop:AllMult}, we use Eq.~\eqref{Eq:coreBra} and compute
    \begin{align*}
        \Delta^\kk_g[z_0,z_1] & =\Delta^\kk_g\big(\ell_e^{\partial z_0}(z_1)+(d_e\Delta^\kk)_{\rho(z_1)}(z_0)\big),
    \end{align*}
    and 
    \begin{align*}
        [\Delta^\kk_gz_0,\Delta^\kk_gz_1] & =\ell_e^{\partial\Delta^\kk_gz_0}(\Delta^\kk_gz_1)+(d_e\Delta^\kk)_{\rho(\Delta^\kk_gz_1)}(\Delta^\kk_gz_0).
    \end{align*}
    As a consequence of Eq.~\eqref{Eq:AlmostCommute}, 
    \begin{eqnarray*} 
    \resizebox{1\hsize}{!}{$\Delta^\kk_g\big(\ell_e^{\partial z_0}(z_1)\big)=\ell_e^{\partial z_0}(\Delta^\kk_gz_1)-\omega_g(\partial z_0,\partial z_1)+(d_g\Delta^\kk)_{\alpha(\partial z_0)}(z_1)-(d_{(g,e)}\Omega)_{(\alpha(\partial z_1),\rho(z_1))}(\partial z_0);$}
    \end{eqnarray*}
    whereas, Eq.~$\textnormal{(M02)}$ in Proposition~\ref{Prop:AllMult} implies
    \begin{align*} 
    \ell_e^{\Delta^\hh\partial z_0}(\Delta^\kk_gz_1) & =\ell_g^{\partial z_0}(\Delta^\kk_gz_1)+(d_{(e,g)}\Omega)_{(\rho(\Delta^\kk_gz_1),0)}(\partial z_0).
    \end{align*}
    Eq.~$\textnormal{(M11)}$ thus follows from $\Delta^\hh\circ\partial=\partial\circ\Delta^\kk$, Eq.~\eqref{Eq:G-equiv} 
    and the differential of $\Delta^\kk_{gh}-\Delta^\kk_g\circ\Delta^\kk_h=\Omega_{(g,h)}\circ\partial$ evaluated at $(\rho_e(\Delta^\kk_g z_1),0_g)\in\gg\oplus T_gG$ 
    and at $(0_g,\rho_e(z_1))\in  T_gG\oplus\gg$.
    \end{proof}

    \begin{proof}[Proof of Theorem~\ref{Theo:MainEqceSuite}]
    Let $A\rightrightarrows\hh$ be a VB-group over $G$ with core $\kk$ and let $\sigma$ be a unit extending splitting of~\eqref{Eq:CoreSeq}. Then, Theorem~\ref{Theo:MainEqce} implies $(\partial,\Delta^\hh,\Delta^\kk,\Omega)$ is a RUTH of $G$ on $\kk[1]\oplus\hh$. If $A\rightrightarrows\hh$ is an LA-group, the source map $s:A\to\hh$ is a Lie algebroid morphism and Proposition~\ref{Prop:MainAUTH} implies that $(\alpha,\Delta,\omega)$ is an AUTH of $\hh$ on $\ker(s)$. Furthermore, as the unit map $u:\hh\to A$ is a Lie algebroid morphism, Proposition~\ref{Prop:uMAP} implies that $\alpha$ and $\omega$ vanish identically at the identity. Since the multiplication $m:A^{(2)}\to A$ is a Lie algebroid morphism as well, Eq.'s~$\textnormal{(M01)}$ and $\textnormal{(M02)}$ in Proposition~\ref{Prop:AllMult} hold and in particular evaluating at $e$ and $(e,g)$, respectively, yields Eq.'s~\eqref{Eq:coreBra} and~\eqref{Eq:ellsuite}. To check that $(\partial,\Delta^\hh,\Delta^\kk,\Omega)$ and $(\rho_e,\alpha,\ell_e,\omega)$ form an LA-matched pair it remains to see that the axioms $(iii)$-$(viii)$ in Definition~\ref{Def:LAMtchdPair} hold. Indeed, Eq.~\eqref{Eq:h-equiv} follows from Proposition~\ref{Prop:Target} as the target map $t:A\to\hh$ is a Lie algebroid morphism; Eq.~\eqref{Eq:G-equiv} follows from Proposition~\ref{Prop:Rho&RUTH} as the multiplication map $m$ is anchored; Eq.~\eqref{Eq:AlmostFlat} follows from Proposition~\ref{Prop:anchorBrackets} as the anchor map $\rho:A\to TG$ induces a Lie algebra homomorphism at the level of sections; Eq.~\eqref{Eq:AlmostRepn} follows from Proposition~\ref{Prop:Rho&RUTH} as well; and lastly, Eq.'s~\eqref{Eq:AlmostCommute} and~\eqref{Eq:MCEq} are respectively Eq.'s~$\textnormal{(M12)}$ and~$\textnormal{(M22)}$ in Proposition~\ref{Prop:AllMult}. 
    
    \noindent Conversely, given an LA-matched pair, Theorem~\ref{Theo:MainEqce} gives a VB-group $G\times(\kk\oplus\hh)\rightrightarrows\hh$ and Proposition~\ref{Prop:MainAUTH} endows $G\times(\kk\oplus\hh)$ with a Lie algeboroid structure for which the source projection is a Lie algebroid morphism. The unit embedding is a Lie algebroid morphism by Proposition~\ref{Prop:uMAP}. Using Eq.'s~\eqref{Eq:coreBra},~\eqref{Eq:h-equiv} and the differential of the defining equation $\Delta^\hh\circ\partial=\partial\circ\Delta^\kk$ evaluated at $\rho_e(z_1)\in\gg$, one sees that $\partial:\kk\to\hh$ defines a Lie algebra homomorphism; therefore, by Proposition~\ref{Prop:Target} and Lemma~\ref{Lemma:ell}, that the target is a Lie algebroid morphism follows from Eq.'s~\ref{Eq:h-equiv} and~\ref{Eq:AlmostFlat}. Using Eq.'s~\eqref{Eq:coreBra},~\eqref{Eq:AnchorRel1@e} and the differential of Eq.~\eqref{Eq:G-equiv} evaluated at $\rho_e(z_1)\in\gg$, one sees that $\rho_e:\kk\to\gg$ defines a Lie algebra homomorphism; hence, invoking Proposition~\ref{Prop:Rho&RUTH}, one concludes that the anchor is a Lie groupoid homomorphism as a consequence of Eq.'s~\eqref{Eq:G-equiv} and~\eqref{Eq:AlmostRepn}. Lastly, by Proposition~\ref{Prop:AllMult} and Lemma~\ref{Lemma:ell}, that the multiplication is a Lie algebroid morphism follows from Eq.'s~\eqref{Eq:coreBra},~\eqref{Eq:ellsuite},~\eqref{Eq:AlmostCommute} and~\eqref{Eq:MCEq}.
    \end{proof}


\subsection{The Lie algebroid structure}\label{ssec-LAStr}
In this section, we study the structure of the Lie algebroid of arrows in an LA-group $A\rightrightarrows\hh$ using the AUTH $(\alpha,\nabla,\omega)$ of $\hh$ on $\ker(s)$ induced by a splitting $\sigma$ of~\eqref{Eq:CoreSeq}. The results in this section parallel those in Subsection~\ref{ssec:2} in the sense of Remarks~\ref{Rmk:CoreDiagram} and~\ref{Rmk:AUTH}. As we explain in Section~\ref{sec:1}, right multiplication yields an isomorphism $\ker(s)\cong G\times\kk$, where $\kk$ is the core. In order, one obtains a representation-like operator $\ell_e:\hh\to\kk^*\otimes\kk$ by acting on constant sections and evaluating at the identity; in symbols, given $y\in\hh$ and $z\in\kk$,
$$\ell_e^y(z):=(\nabla_y(z^R))_e=[y^\sigma,z^R]^1_e.$$
Eq.~\eqref{Eq:ellRep@e} implies $\ell_e$ fails in general to define a representation of $\hh$ on $\kk$, as $d_e\omega:\gg\to\wedge^2\hh^*\otimes\kk$ does not always vanish on $\ll=\rho_e(\kk)\leq\gg$. Similarly, given $y\in\hh$, a sufficient condition for the endomorphism $\ell_e^y$ to be a derivation of $\kk$ is that $d_e\ell:\gg\to\hh^*\otimes\kk^*\otimes\kk$ vanishes at $\ll\leq\gg$. In general, if $\partial:\kk\longrightarrow\hh$ and $\ell_g$ would assemble into a crossed module of Lie algebras for some $g\in G$, $\ell_g$ should define an honest action by derivations of $\hh$ on $\kk$. We conclude, thus, that this is not the case; instead, we have the following result.

\begin{prop}\label{Prop:IsotropyLie2Alg}    
If $A\rightrightarrows\hh$ is an LA-group, the restriction $\gg_e(A)\rightrightarrows\hh$ is a strict Lie 2-algebra (in the sense of \cite{Lie2Alg}). 
\end{prop}
\begin{proof}
    $\gg_e(A)$ is by definition a Lie algebra. Moreover, if $(a_0,a_1)\in\gg_e(A)\times_\hh\gg_e(A)$, $\rho_e(a_0\Join a_1)=\rho_e(a_0)\ast\rho_e(a_1)=0$; therefore, $\gg_e(A)$ defines a subgroupoid of $A\rightrightarrows\hh$. Proposition~\ref{Prop:uMAP} implies it is full. To prove $\gg_e(A)\rightrightarrows\hh$ is a Lie 2-algebra, we let $\mathfrak{k}=\kk\cap\gg_e(A)$ (cf. Lemma~\ref{lemma:uAnchord} and Corollary~\ref{Cor:IsoAt e}) and show that the restriction of the target map $t_e:\mathfrak{k}\longrightarrow\hh$ together with the action $\Lie:\hh\longrightarrow\ggl(\mathfrak{k})$ given by $\Lie_yz:=[u(y),z]_{\gg_e(A)}$ for $y\in\hh$ and $z\in\mathfrak{k}$ defines a crossed module of Lie algebras.

    \noindent First, combining Corollary~\ref{Cor:IsoAt e} and Proposition~\ref{Prop:Target}, it follows that $t_e:\mathfrak{k}\longrightarrow\hh$ is a Lie algebra homomorphism. Given that the Lie bracket $[u(y),z]_{\gg_e(A)}$ is computed using any extension of $u(y)$ and $z$, let $\sigma$ be a unit extending splitting of the core sequence~\eqref{Eq:CoreSeq} and compute $$\Lie_yz=[y^\sigma,z^R]^1_e=\ell_e^y(z).$$
    Since $z\in\mathfrak{k}=\ker(\rho_e)\cap\kk$, it follows from Eq.~\eqref{Eq:ellRep@e} that $\Lie$ indeed defines a representation of $\hh$ on $\mathfrak{k}$. In order, by definition, $\hh$ acts by derivations of $\mathfrak{k}$. The equivariance of the structural map and the infinitesimal Peiffer identity follow from Eq.'s~\eqref{Eq:h-equiv} and~\eqref{Eq:coreBra}, respectively.  
\end{proof}

The fact that $\ell_e$ restricts to a honest action of $\hh$ on $\mathfrak{k}$ parallels Lemma~\ref{lemma:Dela1Indep}; 
indeed, notice that in the proof of Proposition~\ref{Prop:IsotropyLie2Alg}, other than it being unit extending, the splitting was chosen arbitrarily. 
The next results are parallel to Lemma~\ref{lemma:Dela0Indep} and Proposition~\ref{Prop:GpdIsotropies}.
\begin{lemma}\label{lemma:AlphaBarIndep}
    The map $\bar{\alpha}:G\longrightarrow\hh^*\otimes\gg/\ll$ defined by the point-wise composition of $\alpha$ with the quotient map onto $\gg/\ll$ is independent of the splitting. Moreover, for each $g\in G$, $\ker(\bar{\alpha}_g)$ defines a Lie subalgebra of $\hh$.
\end{lemma}
\begin{proof}
    Let $\sigma^0$ and $\sigma^1$ be two splittings of~\eqref{Eq:CoreSeq}. 
    Then, for all $y\in\hh$ and all $g\in G$, $\sigma^1_g(y)-\sigma^0_g(y)\in\ker(s_g)$ and there exists a unique $z\in\kk$ such that $\sigma^1_g(y)-\sigma^0_g(y)=z\Join 0_g$. As a consequence, composing with the anchor yields $$\rho_g(\sigma^1_g(y))-\rho_g(\sigma^0_g(y))=\rho_e(z)\ast 0_g,$$
    and $\rho_g(\sigma^1_g(y))\ast0_{g^{-1}}-\rho_g(\sigma^0_g(y))\ast0_{g^{-1}}\in\ll$ thus proving the independence.

    \noindent For the second statement, let $y_0,y_1\in\ker(\bar{\alpha}_g)$ and pick $z_1,z_2\in\kk$, so that $\alpha_g(y_k)=\rho_e(z_k)$ for $k\in\lbrace0,1\rbrace$.
    Then, from Proposition~\ref{Prop:anchorBrackets}, 
    \begin{align}\label{Eq:SubAlg}
        \alpha_g([y_0,y_1]^0)=\rho_e\big(\ell_g^{y_0}(z_1)-\ell_g^{y_1}(z_0)-[z_0,z_1]-\omega_g(y_0,y_1)\big)\in\ll
    \end{align}
    as claimed.
\end{proof} 
In the statement of the following lemma, we stick to our notation and write $\d_\rho:\wedge^n\gg\otimes\kk\longrightarrow\wedge^{n+1}\gg\otimes\kk$ for the \textbf{Chevalley-Eilenberg differential} of a Lie algebra $\gg$ with values in a \emph{quasi-representation} $\rho:\gg\to\kk^*\otimes\kk$. As it has been the case, $\d_\rho$ will not square to zero unless $\rho$ is a true representation. 

\begin{prop}\label{Prop:ALgbdIsotropies}
    Let $A\rightrightarrows\hh$ be an LA-group over $G$ with core $\kk$ and let $g\in G$. Then, there is a short sequence of Lie algebras
    \begin{eqnarray}\label{Eq:AlgIsoCoreSeq}
    \xymatrix{
    (0) \ar[r] & \mathfrak{k} \ar[r] & \gg_g(A) \ar[r]^{s_g\quad} & \ker(\bar{\alpha}_g) \ar[r] & (0),
    }
    \end{eqnarray}
    where $\mathfrak{k}=\gg_e(A)\cap\kk$ and the injective map is given by right multiplication $z\mapsto z\Join0_g$.
    Moreover, given a splitting $\sigma$ of the core sequence~\eqref{Eq:CoreSeq} of $A\rightrightarrows\hh$, a splitting of~\eqref{Eq:AlgIsoCoreSeq} is determined by a map $Z\in\ker(\bar{\alpha}_g)^*\otimes\kk$ such that $\alpha_g(y)=\rho_e(Z(y))$; such splitting induces an isomorphism of Lie algebras between $\gg_g(A)$ and the semi-direct product of $\ker(\bar{\alpha}_g)$ and $\mathfrak{k}$ with respect to the representation $\rho^Z:=\ell_g-\ad\circ Z$ twisted by the 2-cocycle $\mu^Z:=\omega_g-\d_{\ell_g}Z+\frac{1}{2}[Z,Z]\in\wedge^2\ker(\bar{\alpha}_g)^*\otimes\mathfrak{k}$.
\end{prop}
\begin{proof}
    That the map $\mathfrak{k}\to\gg_g(A)$ is well-defined follows as, given $z\in\mathfrak{k}$, $\rho_g(z\Join0_g)=\rho_e(z)\ast0_g=0_g\in T_gG$. This map is clearly injective as $z\Join0_g=0_g$ implies $z=0_g\Join0_{g^{-1}}=0_e\in\mathfrak{k}$. Moreover, the sequence~\eqref{Eq:AlgIsoCoreSeq} is exact at $\gg_g(A)$ because $s_g(z\Join 0_g)=s_g(0)=0\in\hh$ for all $z\in\mathfrak{k}$ and, conversely, if $a\in\ker(s_g)\cap\gg_g(A)$, $a\Join0_{g^{-1}}\in\kk$ and $\rho_e(a\Join0_{g^{-1}})=\rho_g(a)\ast 0_{g^{-1}}=0\in\gg$. Using a splitting $\sigma$ of the core sequence~\eqref{Eq:CoreSeq} of $A\rightrightarrows\hh$, one verifies that $s_g$ restricted to $\gg_g(A)$ is well-defined and surjective. Indeed, if $a\in\gg_g(A)\subseteq A_g$, then $a$ corresponds to some $(z,y)\in\kk\oplus\hh$ and $\rho(z,y)=\rho_e(z)+\alpha_g(y)=0\in\gg$; therefore, $s_g(a)=y\in\ker(\bar{\alpha}_g)$. For the surjectivity, if $y\in\ker(\bar{\alpha}_g)$, there exists a $z\in\kk$ such that $\alpha_g(y)=\rho_e(z)$ and the element $\sigma_g(y)-z\Join0_g\in\gg_g(A)$.

    \noindent To prove the last statement, note that a map $Z\in\ker(\bar{\alpha}_g)^*\otimes\kk$ such that $\alpha_g(y)=\rho_e(Z(y))$ indeed defines a splitting 
    $y\longmapsto \sigma_g(y)-Z(g)\Join 0_g$. Any such splitting induces a representation $\rho^Z$ by adjoining in $\gg_g(A)$ and a cocycle $\mu^Z\in \wedge^2\ker(\bar{\alpha}_g)^*\otimes\mathfrak{k}$. Letting $y\in\ker(\bar{\alpha}_g)$ and $z\in\mathfrak{k}$, one computes
    \begin{align*}
        \rho^Z_y(z) & =[\sigma_g(y)-Z(g)\Join0_g,z\Join 0_g]_{\gg_g(A)}\Join0_{g^{-1}} \\
            & =[y^\sigma-Z(g)^R,z^R]^1_g\Join0_{g^{-1}}=\ell^y_g(y)-[Z(y),z]=(\ell_g^y-\ad_{Z(y)})(z)\in\mathfrak{k}.
    \end{align*}
    Given $y_0,y_1\in\ker(\bar{\alpha}_g)$, the value of the cocycle $\mu^Z(y_0,y_1)$ is given by the difference between 
    $$[\sigma_g(y_0)-Z(y_0)\Join0_g,\sigma_g(y_1)-Z(y_1)\Join0_g]_{\gg_g(A)}\Join0_{g^{-1}}$$
    and 
    $$\big(\sigma_g([y_0,y_1]^0)-Z([y_0,y_1]^0)\Join0_g\big)\Join0_{g^{-1}}$$
    Since
    \begin{align*}
        [\sigma_g(y_0)-Z(y_0)\Join0_g,\sigma_g(y_1)-Z(y_1)\Join0_g]_{\gg_g(A)} & =[y_0^\sigma-Z(y_0)^R,y_1^\sigma-Z(y_1)^R]^1_g,
    \end{align*}
    we conclude
    \begin{align*}
        \mu^Z(y_0,y_1) & =\omega_g(y_0,y_1)-\ell_g^{y_0}(Z(y_1))+\ell_g^{y_1}(Z(y_0))+Z([y_0,y_1]^0)+[Z(y_0),Z(y_1)] \\
                       & =\omega_g(y_0,y_1)-(\d_{\ell_g}Z)(y_0,y_1)+[Z(y_0),Z(y_1)]\in\mathfrak{k}
    \end{align*}
    as claimed.
\end{proof}
\begin{remark}\label{Rmk:WhyNoTomega}
    Echoing Remark~\ref{Rmk:WhyNotOmega}, because of formula~\eqref{Eq:Bra}, one is tempted to guess that $\omega_g$ is the 2-cocycle twisting the semi-direct product in Proposition~\ref{Prop:ALgbdIsotropies}; however, as we point out in Remark~\ref{Rmk:NotCocycles}, $\omega$ is not necessarily a true cocycle, nor is $\ell_g$ a true representation. In addition, $\omega$ has no reason to be $\mathfrak{k}$-valued. 
\end{remark}

To get a result that parallels Proposition~\ref{Prop:GpdOrbits}, we note the following. 
Since $\F^\kk:=\rho\big(\Gamma(\ker(s))\big)$ is an integrable submodule of $\F^A:=\rho\big(\Gamma(A)\big)$, the leaves of the foliation $\F^\kk$ are submanifolds of the leaves of the Lie algebroid $A$. The following is an easy consequence of Corollary~\ref{Cor:lsubalg}.
\begin{lemma}\label{lemma:leaf1}
    Let $A\rightrightarrows\hh$ be an LA-group over $G$ with core $\kk$. If $L_e^A\subseteq G$ is the leaf of $A$ through the identity element $e\in G$, then $L_e^A$ is a Lie subgroup of $G$ with Lie algebra $\ll=\rho_1(\kk)\leq\gg$. 
\end{lemma}
 
\begin{prop}\label{Prop:ALgbdOrbits}
Let $A\rightrightarrows\hh$ be an LA-group over $G$. If $L^A_g$ is the leaf of $A$ through $g\in G$, then $L^A_e$ acts freely on $L^A_g$. Furthermore, if $L^A_e$ is closed in $G$, the action is also proper and defines a principal bundle.   
\end{prop}
\begin{proof}
    It follows from Corollary~\ref{Cor:lsubalg} that the leaf through $g\in G$ integrating the distribution $\F^\kk$ is the coset of the subgroup of Lemma~\ref{lemma:leaf1} $L_e^A\cdot g\in G/L^A_e$; therefore, $L_e^A\cdot g\subseteq L^A_g$. This induces a natural action of $L_e^A$ on $L^A_g$ by left multiplication; indeed, if $l\in L^A_e$ and $\gamma\in L^A_g$, $l\cdot\gamma\in L^A_e\cdot\gamma\subseteq L^A_\gamma=L^A_g$. This action is obviously free and properness depends on the closedness of $L_e^A\cdot g$ in $L^A_g$, which follows from $L^A_e$ being closed.
\end{proof}

We conclude this section by collecting observations that add to the \emph{symmetry picture} of Remarks~\ref{Rmk:CoreDiagram} and~\ref{Rmk:AUTH}.

\begin{lemma}\label{lemma:kerD}
    Let $A\rightrightarrows\hh$ be an LA-group over $G$ with core $\kk$ and let $\partial:\kk\longrightarrow\hh$ be the restriction of the target map. If $\Delta^1:G\longrightarrow\textnormal{GL}(\ker(\partial))$ is the action of Lemma~\ref{lemma:Dela1Indep}, then $\Delta^1(G)$ is a subspace of Lie algebra automorphisms. Moreover, the restriction of the anchor map $\rho_e:\ker(\partial)\longrightarrow\gg$ is a $G$-equivariant Lie algebra homomorphism.
\end{lemma}
\begin{proof}
    Relation (M11) in Proposition~\ref{Prop:AllMult} restricted to $z_0,z_1\in\ker(\partial)$ implies the first statement. The second statement is in Propositions~\ref{Prop:Rho&RUTH} and~\ref{Prop:anchorBrackets}.
\end{proof}

\begin{prop}\label{Prop:kerD}
    Let $A\rightrightarrows\hh$ be an LA-group over $G$ with core $\kk$ and let $\partial:\kk\longrightarrow\hh$ be the restriction of the target map. If $\Delta^{(1)}:=d_e\Delta^1:\gg\longrightarrow\ggl(\ker(\partial))$ is the derivative of the action of Lemma~\ref{lemma:Dela1Indep} at the identity, then, along with this action, the restriction of the anchor map $\rho_e:\ker(\partial)\longrightarrow\gg$ is a crossed module of Lie algebras.
\end{prop}
\begin{proof}
    The first part of Lemma~\ref{lemma:kerD} implies that $\Delta^{(1)}(\gg)$ is a subspace of derivations; the second implies that $\rho_e:\ker(\partial)\longrightarrow\gg$ is $\gg$-equivariant. The infinitesimal Peiffer equation is Eq.~\eqref{Eq:coreBra} restricted to $z_0,z_1\in\ker(\partial).$
\end{proof}
\begin{theorem}\label{Theo:Btrfls}
    Let $A\rightrightarrows\hh$ be an LA-group over $G$ with core $\kk$ and let $\partial:\kk\longrightarrow\hh$ be the restriction of the target map. If $\mathfrak{k}=\kk\cap\ker(\rho_e)$ and $\ll:=\rho_e(\kk)$, then $\kk$ is a Morita equivalence between the strict Lie 2-algebras associated with the crossed modules $\partial:\mathfrak{k}\longrightarrow\partial\kk$ of Proposition~\ref{Prop:IsotropyLie2Alg} and $\rho_e:\ker(\partial)\longrightarrow\ll$.
\end{theorem}
\begin{proof}
    Since Morita equivalences in the category of strict Lie 2-algebras correspond to \textbf{butterflies} both of whose diagonals are short exact sequence (see \cite{Noohi}), the result is immediate.
\end{proof}
\begin{cor}\label{Cor:abIso&orb}
    In the context of Theorem~\ref{Theo:Btrfls}, $V:=\ker(\partial)\cap\ker(\rho_e)$ is an abelian Lie algebra and $\partial\kk/\partial\mathfrak{k}\cong\ll/\rho_e(\ker(\partial))$ as Lie algebras.
\end{cor}


\section{Applications to integrability}\label{sec:5}
In this section, we use Theorem~\ref{Theo:MainEqceSuite} to catalog extreme examples of LA-groups where part of the structure is simplified. 
More precisely, letting $A\rightrightarrows\hh$ be an LA-group over $G$ with core $\kk$ and $\sigma$ be a unit extending splitting of the core sequence~\eqref{Eq:CoreSeq}, 
Theorem~\ref{Theo:MainEqceSuite} implies there is an LA-matched pair. 
Although different splittings may yield different LA-matched pairs, throughout, 
we have identified pieces that are independent of this choice. 
In particular, if $\gg$ is the Lie algebra of $G$, 
the structural maps $\partial:\kk\to\hh$ and $\rho_e:\kk\to\gg$ are so, 
as well as the respective actions of $\hh$ and $G$ induced on the cohomology spaces of these. 
Letting $\mathfrak{k}=\kk\cap\gg_e(A)$ and $\ll=\rho_e(\kk)$, 
in what follows, we consider the butterflies of Theorem~\ref{Theo:Btrfls}
\begin{eqnarray}\label{Eq:BttrFly}
    \xymatrix{
    \mathfrak{k} \ar[dr]\ar[dd]_{\partial} & & \ker(\partial) \ar[dl]\ar[dd]^{\rho_e} \\
     & \kk \ar[dr]\ar[dl] &  \\
     \partial\kk & & \ll 
    }
\end{eqnarray}
and catalog some of the examples that one gets by trivializing their building components. 
The reader should be warned not to confuse the butterfly of~\eqref{Eq:BttrFly} with the core diagrams of, e.g.~\cite{CoreDiagrams}; 
indeed, in the core diagram both $\hh$ and $\gg$ appear in their entirety, as opposed to the images of the structural maps $\partial$ and $\rho_e$.
${}$\\

\noindent In the upcoming subsections, we assume the following simplifications: 
In Subsection~\ref{ssec-LAIntr}, we consider the extreme cases where either $\gg$ or $\hh$ are trivial. 
This will follow as a consequence of a more general integrability result for LA-groups whose algebroid of arrows is regular (see Theorem~\ref{Theo:IntGps}).  
In Subsection~\ref{ssec-TheWorld}, we study the consequences of the structural maps $\partial$ and $\rho_e$ vanishing identically. 
Lastly, in subsection~\ref{ssec-final}, we study the extreme cases of having a transitive $A$, $\ll=\gg$, 
or a totally intransitive $A$, which implies $\ll=(0)$. 

\subsection{The regular case}\label{ssec-LAIntr}
If one supposes that $\gg=(0)$, $G=\ast$ and $A\rightrightarrows\hh$ reduces to a strict Lie 2-algebra (cf. Proposition~\ref{Prop:IsotropyLie2Alg}). 
Strict Lie 2-algebras are known to be integrable to strict Lie 2-groups \cite{ZhuInt2Alg,Angulo:2021}. 
Due to how symmetric double Lie groups are, 
it is expected that if $\hh=(0)$, $A\rightrightarrows (0)$ integrates to a strict Lie 2-group as well. 
This is indeed the case, as one can deduce from the following result.
\begin{theorem}\label{Theo:IntGps}
     Let $A\rightrightarrows\hh$ be an LA-group over $G$. 
     If $A\to G$ is regular, then $A$ is integrable as a Lie algebroid; 
     moreover, $A\rightrightarrows\hh$ is integrable as an LA-group.
\end{theorem}
\begin{proof}
    Since $\rho_e(A_e)=\rho_e(\kk)=\ll$ and $d_eR_g(\ll)=\rho_g(\ker(s_g))\leq\rho_g(A_g)$ for all $g\in G$, if $A$ is regular, $\F^A=\rho\big(\Gamma(\ker(s))\big)\leq\XX(G)$ and the leaf of $A$ through $g\in G$, $L_g^A$, coincides with the coset of the subgroup $L_e^A\cdot g\in G/L^A_e$ of Lemma~\ref{lemma:leaf1}. Consequently, for all $g\in G$, $\pi_2(L_g^A)=(0)$ and $A$ is integrable. 
    
    \noindent For the latter part of the statement, 
    let $D\rightrightarrows G$ be the source connected and source simply connected integration of $A$. 
    As a consequence of Proposition~\ref{Prop:ALgbdIsotropies}, all isotropy groups have the same topology as do all orbits; 
    therefore, all source fibres are isomorphic, e.g. to the source fibre over the identity $s_G^{-1}(e)$.
    If $H$ is the connected and simply connected Lie group with Lie algebra $\hh$, 
    then the source and the target maps integrate to submersive Lie groupoid morphisms $s_H$ and $t_H$ from $D\rightrightarrows G$ to $H$.
    One assembles thus the fibred-product groupoid $D\times_HD\rightrightarrows G\times G$ whose Lie algebroid is $A\times_\hh A$, 
    and whose typical source fibre is isomorphic to the fibred-product $s_G^{-1}(e)\times_Hs_G^{-1}(e)$. 
    Let $\kk$ be the core of $A\rightrightarrows\hh$. 
    If $C$ is the connected and simply connected Lie group with Lie algebra $\kk$,
    $s_H:s_G^{-1}(e)\to H$ is a principal $C$-bundle over $H$. 
    Consequently, $s_G^{-1}(e)\times_Hs_G^{-1}(e)=t_H^*(s_G^{-1}(e))$ is also a principal $C$-bundle. 
    The long exact sequence in homotopy of the former bundle implies that $\pi_2(s_G^{-1}(e))=(0)$, while that of the latter implies
    that $D\times_\hh D\rightrightarrows G^2$ is the source connected and source simply connected integration of $A\times_\hh A$. 
    Thus, ultimately, there exists a Lie groupoid morphism $D\times_HD\to D$ integrating the multiplication map $m:A\times_\hh A\to A$.
\end{proof}

\begin{cor}\label{Cor:IntGps}
     If $A\rightrightarrows(0)$ is an LA-group, then $A$ is integrable as a Lie algebroid and its integration is a strict Lie 2-group.
\end{cor}
In fact, in the context of Remark~\ref{Rmk:omega=0}, if $\rho(\Img(\sigma))\equiv0$, the hypothesis of Theorem~\ref{Theo:IntGps} is verified and the following result ensues.
\begin{cor}\label{Cor:Int}
    If the core sequence~\eqref{Eq:CoreSeq} of the LA-group $A\rightrightarrows\hh$ admits a splitting that is a Lie algebroid morphism, then $A$ is integrable and its integration is a double Lie group.
\end{cor}
\begin{proof}
    Corollary~\ref{Cor:IntGps} implies $G\times\kk\rightrightarrows(0)$ integrates to $K\rtimes G\rightrightarrows G$; on the other hand, $G\times\hh$ integrates to the trivial bundle of Lie groups $G\times H\rightrightarrows G$. Integrating the core sequence~\eqref{Eq:CoreSeq} yields the result.
\end{proof}
Note that for any splitting of the core sequence, $A$ is regular if and only if $\ker(\bar{\alpha}_g)=\hh$ for all $g\in G$. 
We conclude this section by wondering if there could be Lie theoretic conditions --- such as $\hh$ being simple, --- 
that would imply the regularity of $A$. Consider Propositions~\ref{Prop:IsotropyLie2Alg} and~\ref{Prop:kerD}, 
where, as a consequence of having crossed modules, both $\partial\mathfrak{k}$ and $\rho_e(\ker(\partial))$ are ideals, 
respectively, in $\hh$ and $\gg$. Note that as a consequence, the subgroup of $L_e^A$ integrating $\rho_e(\ker(\partial))$ 
is normal in $G$ and the foliation tangent to the cosets is \emph{multiplicative} (see~\cite{Jotz1} and \cite{DiracLieGps}). 
If $\gg$ is simple, either $\ker(\partial)\leq\gg_e(A)$ or the algebroid of arrows is transitive (see Subsection~\ref{sss-transitiveCore} below).


\subsection{Vacant groupoids, crossed actions and representations}\label{ssec-TheWorld}
LA-groups whose core is trivial are referred to in the literature as \textbf{vacant}. 
In the absence of a core, 
the surviving structure amounts to a pair of actions that verify axioms making them compatible to assemble into a double Lie group. 
Recall that the quasi-actions of a RUTH induce actions on the cohomology groups of the structural map $\partial$ (see Lemmas~\ref{lemma:Dela0Indep} and~\ref{lemma:Dela1Indep}), and that the quasi-action $\alpha$ of an AUTH has $\ll$-valued curvature. 
If both $\partial$ and $\rho_e$ are identically zero, 
aside from the above referred actions, the core turns into an abelian Lie algebra on which the sides act. 
In this case, the actions including the representations, still integrate and assemble into a double Lie group.
We explain this in detail below.

\subsubsection{The case $\kk=(0)$ --- }\label{sss-vacant}  
In this case, diagram~\eqref{Eq:BttrFly} vanishes identically. 
The cotangent groupoids of Poisson Lie groups arguably constitute the class vacant LA-groups that has been studied the most (see, e.g. \cite{Drinfeld,LuWein2}). The double Lie group integrating these is produced --under favorable conditions-- in \cite{LuWein}. We proceed to discuss how their construction generalizes to any given vacant LA-group (cf. also Subsection~\ref{sss-vanishingCurvatures} below).
As there is no core, picking a splitting of the core sequence~\eqref{Eq:CoreSeq} amounts to fixing an isomorphism $\sigma:G\times\hh\to A$.
The induced RUTH and AUTH reduce respectively to group action $\Delta^0:G\to$GL$(\hh)$ and a single Lie algebra action $\alpha:\hh\to\XX(G)$. 
These two actions are further compatible in the sense that the connection $\nabla^\hh$ of $\hh$ on $TG\times\hh$ in~\eqref{Eq:AlmostFlat}, 
as well as the map $\Xi^\hh$ in~\eqref{Eq:AlmostRepn} respectively yield representations. 
Using $\sigma$, $A$ turns into the action Lie algebroid associated to $\alpha$. 
Such algebroids are known to always be integrable, 
but they integrate to action groupoids only when the action is complete. 
If $H\times G\to G:(h,g)\mapsto h\bullet g$ integrates $\alpha$, 
then $H\ltimes G\rightrightarrows G$ integrates $A$ and 
integrating the target map yields an action $H\times G\to H:(h,g)\mapsto h^g$ such that 
\begin{eqnarray}\label{Eq:compatible}
  (h_0h_1)^g=h_0^{h_1\bullet g}h_1^g & \textnormal{ and } & h\bullet g_0g_1=(h^{g_1}\bullet g_0)(h\bullet g_1)
\end{eqnarray}
for all $h, h_0,h_1\in H$ and $g,g_0,g_1\in G$. 
The reader is invited to check that the Eq.'s~\eqref{Eq:compatible} differentiate respectively to Eq.~\eqref{Eq:AlmostFlat} and~\eqref{Eq:AlmostRepn}. 
Since $G$ and $H$ act on each other and the actions verify Eq.~\eqref{Eq:compatible}, the conclusion of \cite[Theorem 3.8.]{LuWein2} holds. 
That is, there is a Lie group structure on $G\times H$, denoted by $G\Join H$, given by $$(g_0,h_0)\Join(g_1,h_1)=\big(g_0(h_0\bullet g_1),h_0^{g_1}h_1\big)$$ for $g_0,g_1\in G$ and $h_0,h_1\in H$, for which $G$ and $H$ are closed subgroups and such that the map $G\times H\to G\Join H:(g, h)\mapsto(g,e)\Join(e,h)$ is a diffeomorphism.
$G\Join H$ forms a double Lie group whose side groupoids are transformation groupoids associated with the actions.
In the case of a Poisson Lie group, $\hh=\gg^*$ and the integrating actions are the so-called \textbf{dressing actions} \cite{LuWein}. 

\subsubsection{The case $\partial\equiv0$, $\ll=(0)$ --- }\label{sss-trivial} Assuming the LA-group non-vacant and both structural maps null 
has as first consequence that $\mathfrak{k}=\kk=\ker(\partial)$ is an abelian Lie algebra which, in keeping with Corollary~\ref{Cor:abIso&orb}, we call $V$. 
Given a splitting of the core sequence~\eqref{Eq:CoreSeq}, 
the quasi-actions of the induced RUTH turn into the two actions $\Delta^0:G\to$GL$(\hh)$ and $\Delta^1:G\to$GL$(V)$. 
In a similar fashion, the induced AUTH gets defined by the Lie algebra action $\alpha:\hh\to\XX(G)$ and the Lie algebra representation $\ell_e:\hh\to\ggl(V)$. 
Along with these, one has the curvatures $\Omega$ and $\omega$ which verify their corresponding cocycle equations $\delta^\Delta\Omega=0$ and $\d_\nabla\omega=0$. 
This data is further subject to four equations: 
As in the case of vacant LA-groups, the connection $\nabla^\hh$ of $\hh$ on $TG\times\hh$ in~\eqref{Eq:AlmostFlat}, 
as well as the map $\Xi^\hh$ in~\eqref{Eq:AlmostRepn} respectively yield representations; 
also, $\d_\nabla\Omega=\delta^\Delta\omega$ and 
\begin{align}\label{Eq:compatible'}
    \Delta^1_{g_0}\circ\ell_{g_1}^y-\ell_{g_0g_1}^y\circ\Delta^1_{g_0}=(d_{g_0}\Delta^1)_{\alpha(\Delta^0_{g_1}y)_{g_0}}
\end{align}
for all $g_0,g_1\in G$ and $y\in\hh$. 
Reinterpreting the exact sequence~\eqref{Eq:CoreSeq} as a short exact sequence of Lie algebroids, 
where $\hh\ltimes G$ is the action Lie algebroid associated to $\alpha$ and $G\times V$ is abelian, 
$A$ is determined by the cohomology class $[\omega]\in H^2_\nabla(\hh\ltimes G,G\times V)$. 
Assuming that the action $\alpha$ is complete, one integrates $\hh\ltimes G$ to $H\ltimes G\rightrightarrows G$, with $H$ connected and simply connected. 
Now, $H\times G$ is Hausdorff, and, since the source fibres are all isomorphic to the Lie group $H$, they have vanishing second cohomology groups; 
hence, one can invoke \cite[Theorem 5]{VanEstC} and conclude that $A$ is integrable as well. 
More precisely, $A$ is integrated by the Lie groupoid extension determined by the unique $[\theta]\in H^2_{\textnormal{Gpd}}(H\ltimes G,G\times V)$ 
such that its image under the van Est map $\Phi[\theta]$ coincides with $[\omega]$. 
In building the semi-direct product, the representation $\nabla$ of the AUTH integrates to a representation $\lambda:(H\ltimes G)\times V\to V$. 
Integrating the target map yields again an action $H\times G\to H$ that is compatible with the action integrating $\alpha$ in the sense of~\eqref{Eq:compatible}. 
These structures are intertwined by the formulas $\lambda(h,g_0g_1)=\lambda(h^{g_1},g_0)$ and $\Delta^1_{g_0}\circ\lambda(h,g_1)=\lambda(h,g_0g_1)\circ\Delta^1_{h^{g_1}\bullet g_0}$, 
which differentiate respectively to Eq.'s~\eqref{Eq:ellsuite} and~\eqref{Eq:compatible'}. 
Furthermore, fixed $(g_0,g_1)\in G^2$, there is a $V$-valued 1-form $\Omega^R_{(g_0,g_1)}\in\Omega^1(H,V)$ 
given by right-translating $\Omega_{(g_0,g_1)}\in\hh^*\otimes V$. 
As $H$ is simply connected, one may identify it with the space of paths starting at the identity (see, e.g.~\cite{CF2}), 
and define $\Theta:H\times G^2\to V$ by
$$\Theta(h;g_0,g_1):=\int_h\Omega_{(g_0,g_1)}^R,\textnormal{ for }h:[0,1]\longrightarrow H.$$
The defining property of $\Theta$ is that $\ddt{\tau}\Theta(\exp_H(\tau y);g_0,g_1)=\Omega_{(g_0,g_1)}(y)$. 
One easily computes that it follows from $\delta^\Delta\Omega=0$ that $\Theta$ is a 2-cocycle of $H\rtimes G\rightrightarrows H$ with values on $H\times V$, 
defining thus an extension whose space of arrows is that of the integration of $A$. 
Ultimately, these structures assemble into a double Lie groupoid as a consequence of $\delta^\lambda\Theta=\delta^{\Delta^1}\theta$, 
which obviously differentiate to $\d_\nabla\Omega=\delta^\Delta\omega$. 
Note that in this class of examples, without assuming the completeness of the action $\alpha$, 
the integrating groupoid of $\hh\ltimes G$ must be Hausdorff and must have source fibres with trivial second cohomology 
for this reasoning to work.

\subsubsection{The case $\Omega\equiv0$, $\omega\equiv0$ --- }\label{sss-vanishingCurvatures} 
There is a third situation in which the quasi-actions turn into honest actions; namely, if one is allowed to pick a splitting for which both curvatures vanish. In general, let
    \begin{equation}\label{Eq:LA-matchdPair}  
        (\partial:\mathfrak{c}\to\hh,\Delta^\hh,\Delta^{\mathfrak{c}},\Omega\quad;\quad\rho_e:\mathfrak{c}\to\gg,\alpha,\ell_e,\omega)
    \end{equation}
    be the LA-matched pair corresponding to a unit extending splitting of the core sequence of $A\rightrightarrows\hh$, whose existence is asserted by Theorem~\ref{Theo:MainEqceSuite}. Let now   
    \begin{equation}\label{Eq:MatchdPair}  
        (\partial:\mathfrak{c}\to\hh,\nabla^{\gg\hh},\nabla^{\gg\mathfrak{c}},R_\gg\quad;\quad\rho_e:\mathfrak{c}\to\gg,\nabla^{\hh\gg},\nabla^{\hh\mathfrak{c}}=\ell_e,R_\hh)
    \end{equation}
    be the matched pair of representations up to homotopy of Lie algebras (for notation, cf. \cite[Definition 3.1]{MackJotzRaj}) obtained by \textit{differentiation} of~\eqref{Eq:LA-matchdPair}, as explained in Remark~\ref{Rmk:MatchedPairs}. Using this data, the two-term complex 
    $$\ell_1:=\rho_e\oplus\partial:\mathfrak{c}\longrightarrow\gg\oplus\hh$$
    can be endowed with brackets
    \begin{align}\nonumber
        \ell_2(x_0,y_0;x_1,y_1) & = ([x_0,x_1]_\gg+\nabla^{\hh\gg}_{y_0}x_1-\nabla^{\hh\gg}_{y_1}x_0; [y_0,y_1]_\hh+\nabla^{\gg\hh}_{x_0}y_1-\nabla^{\gg\hh}_{x_1}y_0), \\ \label{Eq:Ell}
        \ell_2(x,y;z) & = \nabla^{\gg\mathfrak{c}}_{x}z+\nabla^{\hh\mathfrak{c}}_{y}z, \\ \nonumber
        \ell_3(x_0,y_0;x_1,y_1;x_2,y_2) & = R_\gg(x_0,x_1)y_2+R_\hh(y_0,y_1)x_2+\textnormal{c.p.}, 
    \end{align}
    for $x_0,x_1,x_2,x\in\gg$, $y_0,y_1,y_2,y\in\hh$, and $z\in\mathfrak{c}$. As it turns out, the equations these brackets need to verify in order to assemble into a two-term $L_\infty$-algebra are \emph{precisely} the equations defining a matched pair.
    \begin{theorem}\label{Theo:InfEquce}
        The brackets~\eqref{Eq:Ell} define a two-term $L_\infty$-algebra if and only if~\eqref{Eq:MatchdPair} is a matched pair.
    \end{theorem}
    
    Suppose now that $\Omega$ and $\omega$ in~\eqref{Eq:LA-matchdPair} vanish identically, so that $R_\gg$ and $R_\hh$ vanish as well. Since crossed modules of Lie algebras \emph{coincide} (cf. Eq.~\eqref{Eq:coreBra}) with two-term $L_\infty$-algebras whose 3-bracket $\ell_3$ vanishes, $\mathfrak{c}\to\gg\oplus\hh$ carries the structure of a crossed module of Lie algebras. Suppose the latter crossed module of Lie algebras integrates to a crossed modules of Lie groups $C\to G\Join H$ \cite{ZhuInt2Alg}; hence, $D:=C\rtimes(G\Join H)\rightrightarrows G\Join H$ is a strict Lie 2-group in the sense of Baez~\cite{Lie2Gps}. $D$ also carries the structure of a double groupoid integrating $A$, where the side groupoids are transformation groupoids associated to the integrated actions. Pressumably, as in \cite{LuWein2}, one can build a double groupoid even when the actions are not complete.

\subsection{Atiyah algebroids and bundles of Lie 2-algebras}\label{ssec-final}
In this section we study the two extreme cases when the structural map $\rho_e:\kk\to\gg$ 
is either surjective, so that $\ll=\gg$, or 
vanishes identically, so that $\mathfrak{k}=\kk$ and $\ll=(0)$. 
In the former case, the Lie algebroid is regular; 
hence, as established in Theorem~\ref{Theo:IntGps}, it is integrable even as an LA-group. 
What is new, however, is that due to its being transitive, 
its integrating Lie groupoid can be explicitly given as the gauge groupoid of a principal bundle. 
In the latter case, $\hh$ acts on $G$ irrespective of the splitting and 
the Lie algebroid can thus be integrated in two steps. 

\subsubsection{The case $\ll=\gg$ --- }\label{sss-transitiveCore}
In this case, the Lie algebroid $A\to G$ is transitive and since $\pi_2(G)=(0)$, 
$A$ is the Atiyah algebroid of a principal bundle $q:P\to G$. 
Let $\gg_G(A)$ be the isotropy bundle and consider the map of exact sequences of vector bundles  
\begin{eqnarray*}
    \xymatrix{
        (0) \ar[r] & \gg_G(A) \ar[r]\ar[d] & A \ar[r]^\rho \ar[d]_s & TG \ar[r]\ar[d]^{ds} & (0) \\
        (0) \ar[r] & G\times\hh \ar[r] & G\times\hh \ar[r] & (0) \ar[r] & (0) 
    }
\end{eqnarray*}
From the associated long exact sequence, 
one sees that the unit extending splitting of the core sequence~\eqref{Eq:CoreSeq} can be chosen to be $\gg_G(A)$-valued. 
In so, the associated LA-matched pair has $\alpha\equiv0$ and $\mathfrak{k}$-valued curvatures $\Omega$ and $\omega$. 
In this case, $\omega_g\in\wedge^2\hh^*\otimes\mathfrak{k}$ is indeed a 2-cocycle for each $g\in G$ (cf. Remark~\ref{Rmk:NotCocycles}), 
and, because all isotropy Lie algebras must be isomorphic, $\gg_g(A)\cong\gg_e(A)$ for each $g\in G$. 
As a consequence, there exists $\beta\in C(G,\hh^*\otimes\mathfrak{k})$ 
such that $\omega_g=\d_{\ell_e}\beta_g+\frac{1}{2}[\beta_g,\beta_g]$ and 
$$[\beta_g(y),z]=\ell_g^yz-\ell_e^yz$$
for all $y\in\hh$ and $z\in\mathfrak{k}$.
Since $\gg_e(A)\rightrightarrows\hh$ is a strict Lie 2-algebra by Proposition~\ref{Prop:IsotropyLie2Alg}, 
it integrates to a strict Lie 2-group $\G_e(D_G)\rightrightarrows H$. 
Now $P$ is a principal $\G_e(D_G)$-bundle, 
and the groupoid $P\times_{\G_e(D_G)}P\rightrightarrows G$ integrates $A$.

In the particular case when the structural map $\partial$ is surjective as well, 
the integrating double Lie group appears determined by its core diagram as exposed in \cite{Jotz2}.

\subsubsection{The case $\mathfrak{k}=\kk$ --- }\label{sss-isotropicCore}
In this case, the core Lie algebra is a subalgebra of the isotropy. 
Because of Remark~\ref{Rmk:AUTH}, the vanishing of $\rho_e$ implies that for any splitting, 
the induced quasi-action $\alpha:\hh\to\XX(G)$ is a Lie algebra action. 
Assuming once more that this action is complete, the Lie algebrod $\hh\ltimes G$ 
integrates to $H\ltimes G\rightrightarrows G$. 
The subalgebroid $\ker(s)$ is now a bundle of Lie algebras whose fibre is the Lie algebra $\kk$; 
consequently, it integrates to the bundle of groups $G\times C\rightrightarrows G$. 
The integration of the Lie algebroid $A$ depends now on \emph{integrating the AUTH} $(\alpha,\nabla,\omega)$. 
More precisely, one should build a quasi-action of $H\ltimes G\rightrightarrows G$ on $G\times C\to G$ that integrates $\nabla$, 
and whose failure to yield an action is controlled by a cochain $C^2(H\ltimes G,C)$ that differentiates to $\omega$. 
This in general can be challenging; however, if the action $\alpha$ is zero, the algebroid is but a bundle of Lie algebras, 
and being regular, it is integrable by Theorem~\ref{Theo:IntGps}. 
Although, in general, the Lie algebra structure will vary from fibre to fibre, 
we think of each fibre as a deformation of the isotropy Lie 2-algebra of Proposition~\ref{Prop:IsotropyLie2Alg}; 
indeed, the integrating double Lie group will be a bundle of Lie groups each of whose fibres is a deformation of the 
strict Lie 2-group integrating the isotropy Lie algebra over the identity.


\bibliographystyle{plain}
\bibliography{template}

\end{document}